\documentclass[11pt]{amsart}
\usepackage{amsmath}
\usepackage{amssymb}
\usepackage{amsfonts}
\usepackage{amsthm}
\usepackage{enumerate}
\usepackage[mathscr]{eucal}

\newtheorem{theorem}{Theorem}[section]
\newtheorem{lemma}[theorem]{Lemma}
\newtheorem{corollary}[theorem]{Corollary}

\theoremstyle{definition}

\theoremstyle{remark}
\newtheorem{remark}[theorem]{Remark}

\numberwithin{equation}{section}

\def\R{{\mathbb R}}
\def\Z{{\mathbb Z}}

\def\intslash{\rlap{\kern  .32em $\mspace {.5mu}\backslash$ }\int}
\def\qsl{{\rlap{\kern  .32em $\mspace {.5mu}\backslash$ }\int_{Q_x}}}

\def\S{\mathbb S}
\def\C{\mathcal C}

\def\N{\mathbb N}

\def\emph#1{{\it #1 }}

\def\pari{\partial}
\def\Ga{\Gamma}
\def\ga{\gamma}

\def\supp{{\text{\rm supp}}}

\def\inn#1#2{\langle#1,#2\rangle}

\def\card{\text{\rm card}}
\def\lc{\lesssim}

\def\pv{\text{\rm p.v.}}
\def\alp{\alpha}

\def\del{\delta}             
\def\eps{\varepsilon}

\def\tet{\theta}

\def\lam{\lambda}            \def\Lam{\Lambda}

\def\om{\omega}              \def\Om{\Omega}
\def\fr{\frac}
\newcommand{\Be}{\begin{equation}}
\newcommand{\Ee}{\end{equation}}
\newcommand{\Bes}{\begin{equation*}}
\newcommand{\Ees}{\end{equation*}}
\newcommand{\Bsp}{\begin{split}}
\newcommand{\Esp}{\end{split}}
\newcommand{\Bm}{\begin{multline}}
\newcommand{\Em}{\end{multline}}
\newcommand{\Bea}{\begin{eqnarray}}
\newcommand{\Eea}{\end{eqnarray}}
\newcommand{\Beas}{\begin{eqnarray*}}
\newcommand{\Eeas}{\end{eqnarray*}}
\newcommand{\Benu}{\begin{enumerate}}
\newcommand{\Eenu}{\end{enumerate}}
\newcommand{\Bi}{\begin{itemize}}
\newcommand{\Ei}{\end{itemize}}

\begin{document}

\title[Weak type (1,1) bound criterion]{Weak type (1,1) bound criterion for singular integral with rough kernel and its applications}

\author{Yong Ding}
\author{Xudong Lai}
\address{\textbf{Yong Ding}:
School of Mathematical Sciences,
         Beijing Normal University,
Laboratory of Mathematics and Complex Systems (BNU), Ministry of Education,
Beijing, 100875,
       People's Republic of China}
\email{dingy@bnu.edu.cn}

\thanks {The work is supported by NSFC (No.11371057, No.11471033, No.11571160), SRFDP (No.20130003110003) and the Fundamental Research Funds for the Central Universities (No.2014KJJCA10).}


\address{\textbf{Xudong Lai}(Corresponding Author):
Institute for Advanced Study in Mathematics, Harbin Institute of Technology, Harbin, 150001, People's Republic of China
\&
School of Mathematical Sciences,
         Beijing Normal University,
         Laboratory of Mathematics and Complex Systems (BNU), Ministry of Education,
Beijing, 100875,
       People's Republic of China}
\email{xudonglai@mail.bnu.edu.cn}
\thanks{Xudong Lai is the corresponding author.}


\date{June 18, 2017}



\begin{abstract}
In this paper, a weak type (1,1) bound criterion is established for singular integral operator with rough kernel. As some applications of this criterion, we show some important operators with rough kernel in harmonic analysis, such as Calder\'on commutator, higher order Calder\'on commutator, general Calder\'on commutator, Calder\'on commutator of Bajsanski-Coifman type and general singular integral of Muckenhoupt type, are all of weak type (1,1).
\end{abstract}

\maketitle


\section{Introduction}

Singular integral theory is a fundamental and important topic in harmonic analysis. It is intimately connected with the study of complex analysis and partial differential equations. Real variable methods of singular integral for higher dimension were original by A. P. Calder\'on and A. Zygmund \cite{CZ52} in the 1950's. Later, large numbers of works are developed in this area. Despite the intensive research over the last six decades, there are still many problems in the theory of singular integral which remain open and deserve to be explored further. For example, there is no general $L^1$ theory of rough singular integral, singular integral along curves and Radon transforms (see \cite{Ste93}).

It is well known that the $L^1$ boundedness is not true for many integral operators in harmonic analysis, such as Hilbert transform, Riesz transforms, Hardy-Littlewood maximal operator, and so on. As a substitution, we consider the weak type (1,1) bound and use interpolation and dual argument, we can get all $L^p$ bound for $1<p<+\infty$. So it is an \emph{important problem} to establish weak type (1,1) boundedness in the $L^1$ theory of singular integral operator and maximal operator. Usually, the weak type (1,1) bound can be established by using the classical Calder\'on-Zygmund decomposition if its kernel has enough smoothness. However, if the kernel is rough, then the standard Calder\'on-Zygmund theory cannot be applied directly. In fact it is a quite difficult problem to prove the weak type (1,1) boundedness of the integral operator with rough kernel. We refer to see the nice works by M. Christ \cite{Chr88}, M. Christ and J. Rubio de Francia \cite{CR88}, M. Christ and C. Sogge \cite{CS88}, S. Hofmann \cite{Hof89}, A. Seeger \cite{See96} \cite{See14}, P. Sj\"ogren and F. Soria \cite{SS97} and Tao \cite{Tao99} about this topic.

However, the papers mentioned above are considered for some special operators. In this paper, we are going to study the general $L^1$ theory of rough singular integral operator. More precisely, we try to give a criterion that could deal with weak type (1,1) boundedness of a class of singular integrals with non-smooth kernel.

Before state our main result, let us firstly give our motivations from some basic examples. The first example is \emph{singular integral with convolution homogeneous kernel}. Suppose $\Om$ is a function defined on $\R^d\setminus\{0\}$ satisfying
\begin{equation}\label{e:2Home}
\Om(rx')=\Om(x'),\text{ for any $r>0$ and $x'\in\S^{d-1}$},
\end{equation}
\begin{equation}\label{cance condi}
\int_{\S^{d-1}}\Om(\tet)d\tet=0
\end{equation}
and
\begin{equation}\label{int condi}
\Om\in L^1(\S^{d-1}),
\end{equation}
where and in the sequel, $d\tet$ denotes the surface measure of $\S^{d-1}$. Then it is easy to see that the following singular integral is well defined for $f\in C_c^\infty(\R^d)$,
\Be\label{e:9Tsin}
Tf(x)=\pv\int_{\R^d}\fr{\Om(x-y)}{|x-y|^d}f(y)dy.
\Ee
In 1956, Calder\'on and Zygmund \cite{CZ56} gave the $L^p$ boundedness of $T$.
\vskip0.2cm
\noindent
\textbf{Theorem A}\ {\rm (\cite{CZ56}).}\ {\it Suppose that $\Omega$ satisfies the conditions \eqref{e:2Home} and \eqref{int condi}, then the singular integral $T$ defined in \eqref{e:9Tsin} extends to a bounded operator on $L^p(\R^d)\,(d\ge2)$ for $1<p<\infty$ if $\Om$ satisfies
one of the following conditions:

{\rm (i)}\ $\Om$ is odd;

{\rm (ii)}\ $\Om$ is even and $\Om\in L\log^+L(\S^{d-1})$ satisfies \eqref{cance condi}.}
\vskip0.2cm

For the case $p=1$, it is a very difficult problem to show that $T$ is of weak type (1,1). In 1988, M. Christ and Rubio de Francia \cite{CR88}
and in 1989, S. Hofmann \cite{Hof89} independently gave weak type (1,1) boundedness  of $T$ for $d=2$. Later, in 1996, A. Seeger \cite{See96} established the weak type (1,1) boundedness  of $T$ for all dimension $d\ge2$. Now let us sum up their nice results as follows.
\vskip0.2cm
\noindent
\textbf{Theorem B.}\ {\it Suppose that $\Omega$ satisfies the conditions \eqref{e:2Home}, \eqref{cance condi} and \eqref{int condi}.

{\rm (i)\ \ (see \cite{CR88}).}\ If $\Omega\in L\log^+\!\!L(\S^1)$, $T$ is of weak type $(1,1)$ for $d=2$. In an unpublished paper, M. Christ and Rubio de Francia pointed out that they succeeded proving similar results hold also for $d\le5$;

{\rm (ii)\ \ (see \cite{Hof89}).}\ If $\Omega\in L^q(\S^1)(1<q\leq\infty)$, $T$ is of weak type $(1,1)$ for $d=2$;

{\rm (iii)\ \ (see \cite{See96}).}\ If $\Omega\in L\log^+\!\!L(\S^{d-1})$, $T$ is of weak type $(1,1)$ for $d\ge2$.}

\vskip0.2cm

The second example is \emph{Calder\'on commutator} introduced by A. P. Calder\'on in his famous paper \cite{Cal65}, which is defined by
\Be\label{TA}
T_{\Om,A} f(x)=\pv\int_{\R^d}\fr{\Om(x-y)}{|x-y|^d}\cdot\fr{A(x)-A(y)}{|x-y|}\cdot f(y)dy,
\Ee
where $A\in Lip(\R^d)$, the class of Lipschitz functions.
\vskip0.2cm
\noindent
\textbf{Theorem C}\ {\rm (\cite{Cal65} or see \cite{CCal75}).}\ {\it Let $d\ge2$. Suppose that $\Omega$ satisfies the conditions \eqref{e:2Home} and \eqref{int condi}, then the commutator $T_{\Om,A}$ maps $L^p(\R^d)$ to itself for $1<p<\infty$ if $\Om$ satisfies
one of the following conditions:

{\rm (i)}\ $\Om$ is even;

{\rm (ii)}\ $\Om\in L\log^+L(\S^{d-1})$ is odd and satisfies
\begin{equation}\label{e:km}
\int_{\S^{d-1}}\Om(\tet)\tet^\alpha d\tet=0,\quad \text{for all}\ \ \alpha\in \Z_+^d\ \ \text{with}\ \ |\alp|=1.
\end{equation}
Here and in the sequel, $\alpha=(\alpha_1,\cdots,\alpha_d)\in \Z_+^d$ is a multi-indices , $|\alpha|=\sum_{j=1}^d\alpha_j$ and $x^\alpha=\prod_{i=1}^dx_i^{\alpha_i}$ when $x\in\R^d$.}
\vskip0.2cm

For a long time, \emph{an open problem} is that whether
{Calder\'on commutator} $T_{\Om,A}$  is of weak type (1,1)  if $\Om$ satisfies \eqref{e:2Home}, \eqref{e:km} and $\Om\in L\log^+L(\S^{d-1})$. In Section \ref{s:92l}, we will give a confirm answer to this problem as an application of our main result.

By careful observation of {singular integral with homogeneous kernel} in \eqref{e:9Tsin} and {Calder\'on commutator} in \eqref{TA}, we conclude that singular integrals in \eqref{e:9Tsin} and \eqref{TA} can be formally rewritten in the following way,
\Be\label{e:8T}
T_\Om f(x)=\pv\int_{\R^d}\Om(x-y)K(x,y)f(y)dy
\Ee
where $\Om$ satisfies \eqref{e:2Home}, \eqref{int condi} and $K$ satisfies
\Be\label{e:8kb}
|K(x,y)|\leq \fr{C}{|x-y|^d},
\Ee
and the regularity conditions: for a fixed $\del\in(0,1]$,
\Be\label{e:8kr}
\begin{split}
&|K(x_1,y)-K(x_2,y)|\leq C\fr{|x_1-x_2|^\del}{|x_1-y|^{d+\del}},\ \ \ \ |x_1-y|>2|x_1-x_2|,\\
&|K(x,y_1)-K(x,y_2)|\leq C\fr{|y_1-y_2|^\del}{|x-y_1|^{d+\del}},\ \ \ \ |x-y_1|>2|y_1-y_2|.
\end{split}
\Ee

In this paper, we are interested in when $T_\Om$ is of weak type (1,1). Our main result is the following.
\begin{theorem}\label{t:9main}
Suppose $K$ satisfies \eqref{e:8kb} and \eqref{e:8kr}. Let $\Om$ satisfy \eqref{e:2Home} and $\Om\in L\log^+L(\S^{d-1})$. In addition, suppose $\Om$ and $K$ satisfy some appropriate cancellation conditions such that $T_\Om f(x)$ in \eqref{e:8T} is well defined for $f\in C_c^\infty(\R^d)$ and extends to a bounded operator on $L^2(\R^{d})$ with bound $C\|\Om\|_{L\log^+L}$. Then for any $\lam>0$, we have
$$\lam m(\{x\in\R^{d}:|T_\Om f(x)|>\lam\})\lc \C_\Om\|f\|_1,$$
where $\C_\Om$ is a finite constant which depends on $\Om$ (see the definition in \eqref{e:7constantom}).
\end{theorem}

It should be pointed out that it is difficult to assume uniform cancellation conditions of $\Om$  in our main result, since it is dependent of $K(x,y)$, such as the conditions \eqref{cance condi} and \eqref{e:km}.
Essentially, in the theory of singular integral, the cancellation conditions of $\Om$ play a key role in proving the $L^2$ boundedness of a singular integral with homogeneous kernel. However, in the present paper, the cancellation conditions actually do not need to be used in our proof of weak type (1,1) boundedness of the singular integral once it is of strong type (2,2).

Note that the conditions in Theorem \ref{t:9main} are easily verified, therefore  Theorem \ref{t:9main} gives a weak type (1,1) bound criterion, which has its own interest in the theory of singular integral. In fact, one will see that applying Theorem \ref{t:9main}, some important and interesting integral operators in harmonic analysis, such as the famous Calder\'on commutator, higher order Calder\'on commutator, general Calder\'on commutator, Calder\'on commutator of Bajsanski-Coifman type and general singular integral of Muckenhoupt type are all of weak type (1,1), see Section \ref{s:92l} for more details.

Since the kernel $\Om(x-y)K(x,y)$ of $T_\Om$ is non-smooth for $\Om\in L\log^+L(\S^{d-1})$, the standard Calde\'on-Zygmund theory can not be applied to proving the weak (1,1) boundedness of $T_\Om$. When the dimension $d=2$, M. Christ and Rubio de Francia \cite{CR88}
or S. Hofmann \cite{Hof89}, used the $TT^*$ method to get the weak type (1,1) bound for rough singular integral operator defined in \eqref{e:9Tsin}. The $TT^*$ method was original by C. Fefferman \cite{Fef70} (see \cite{GH12}, \cite{DL15}, \cite{See96}, \cite{See14} and \cite{DL2} for more applications in singular integrals).
However, for the higher dimensions this method may not be useful. In this paper, our strategy to prove Theorem \ref{t:9main} is based on partly the nice ideas in \cite{See96}. More precisely, we use the microlocal decomposition of the kernel and some $TT^*$ argument in $L^2$ estimate in one part (see the proof of Lemma \ref{l:L^2} in Section \ref{s:933}), which is similar to \cite{See96}. For the other part, we inset a
 multiplier operator of weak type (1,1) with a controllable bound so that the problem can be reduced to $L^1$ estimates of some oscillatory integrals (see the proof of Lemma \ref{l:L^1} in Section \ref{s:94}).
Since $T_\Om$ is a non-convolution operator, the proof in this part is more complicated and we can not apply the properties of multiplier to oscillatory integrals. Thus we have to estimate the kernel of oscillatory integrals directly by using the method of stationary phase.

This paper is organized as follows.
In Section \ref{s:92}, we complete the proof of Theorem \ref{t:9main} based on some lemmas, their proofs will be given in Section \ref{s:93} and Section \ref{s:94}. In Section \ref{s:92l}, we give some important applications of Theorem \ref{t:9main}. Some open problems are listed in Section \ref{s:76}. Throughout this paper, the letter $C$ stands for a positive constant which is independent of the essential variables and not necessarily the same one in each occurrence. Sometimes we use $C_N$ to emphasize the constant depends on $N$. $A\lc B$ means $A\leq CB$ for some constant $C$. $A\approx B$ means that $A\lc B$ and $B\lc A$. For a set $E\subset\R^d$, we denote by  $|E|$ or $m(E)$ the Lebesgue measure of $E$. We denote by $\mathcal{F}f$ or $\hat{f}$ the Fourier transform
of $f$ which is defined by
$$\mathcal{F}f(\xi)=\int_{\R^d} e^{-i\inn{x}{\xi}}f(x)dx.$$
 $\Z_+$ denote the set of all nonnegative integers and $\Z_+^d=\Z_+\times \cdots\times \Z_+.$
Moreover, $\|\Om\|_q:=\big(\int_{\S^{d-1}}|\Om(\tet)|^qd\tet\big)^{\fr{1}{q}}$ and  $\|\Om\|_{L\log^+\!\! L}:=\int_{\S^{d-1}}|\Om(\tet)|\log(2+|\Om(\tet)|)d\tet$.
\vskip1cm

\section {Proof of Theorem \ref{t:9main}}\label{s:92}

In this section we give the proof of Theorem \ref{t:9main} based on some lemmas, their proofs will be given in Section \ref{s:93} and Section \ref{s:94}.

We only focus on dimension $d\geq2$. Let $\Om\in L\log^+L(\S^{d-1})$ with $\|\Om\|_{L\log^+L}<+\infty$. Set the constant
\Be\label{e:7constantom}
\mathcal{C}_\Om=\|\Om\|_{L\log^+L}+\int_{\S^{d-1}}|\Om(\tet)|\big(1+\log^+({|\Om(\tet)|}/{\|\Om\|_{1}})\big)d\tet,
\Ee
where $\log^+a=0$ if $0<a<1$ and $\log^+a=\log a$ if $a\geq1$. Since $\|\Om\|_{L\log^+L}<+\infty$, one can easily check that $\mathcal{C}_\Om$ is a finite constant.
For  $f\in L^1(\R^d)$ and $\lam>0$, using the Calder\'on-Zygmund decomposition at level $\fr{\lam}{\C_\Om}$,
we have the following conclusions (cf. see \cite{Ste93} for example):
\begin{enumerate}[\quad (cz-i)]
\rm\item $f=g+b$;
\item  $\|g\|^2_{2}\lc\lambda\|f\|_{1}/{\C_\Om}$;
\item  $b=\sum_{Q\in \mathcal{Q}} b_Q$, $\supp b_Q\subset Q$, where $\mathcal{Q}$ is a countable set of disjoint dyadic cubes;
\item Let $E=\bigcup_{Q\in \mathcal{Q}} Q$, then $m(E)\lc {\lambda}^{-1}\C_\Om\|f\|_{1}$;
\item $\int b_Q=0$ for each $Q\in \mathcal{Q}$ and $\|b_Q\|_{1}\lc \fr{\lambda}{\C_\Om} |Q|$, so $\|b\|_{1}\lc\|f\|_{1}$ by  (cz-iii) and (cz-iv);
\end{enumerate}

 By the property (cz-i), we have
\[m(\{x: |T_{\Om}f(x)|>\lambda\})\leq m\big(\{x:|T_{\Om}g(x)|>\lambda/2\}\big)+m\big(\{x:|T_{\Om}b(x)|>\lambda/2\}\big).\]
Hence, by Chebyshev's inequality, the fact $T_\Om$ is bounded on $L^2(\R^{d})$ with bound $C\|\Om\|_{L\log^+L}$ and property (cz-ii), we get
\[m(\{x\in\R^d:|T_{\Om}g(x)|>\lambda/2\})\leq 4\|T_{\Om}g\|_{2}^2/{\lambda ^2}\lc{\lambda^{-2}}(\|\Om\|_{L\log^+L}\|g\|_{2})^2\lc{\lambda}^{-1}\C_\Om\|f\|_{1}.\]
For $Q\in \mathcal{Q}$, denote by $l(Q)$ the side length of cube $Q$. For $t>0$, let $tQ$ be the cube with the same center of $Q$ and $l(tQ)=tl(Q)$.
Set $E^*=\bigcup_{Q\in \mathcal{Q}}2^{200}Q$. Then
\[m(\{x\in\R^d:|T_{\Om}b(x)|>\lambda/2\})\leq m(E^*)+m(\{x\in (E^*)^c:|T_{\Om}b(x)|>\lambda/2\}).\]
By the property (cz-iv), the set $E^*$ satisfies
\begin{equation*}
m(E^*)\lc m(E)\lc{\lambda}^{-1}\C_\Om\|f\|_{1}.
\end{equation*}
Thus, to complete the proof of Theorem \ref{t:9main}, it remains to show
\begin{equation}\label{e:weak}
m(\{x\in (E^*)^c:|T_{\Om}b(x)|>\lambda/2\})\lc {\lambda}^{-1}\C_\Om\|f\|_{1}.
\end{equation}

Denote $\mathfrak{Q}_k=\{Q\in \mathcal{Q}:\, l(Q)=2^k \}$ and let $B_k=\sum\limits_{Q\in\mathfrak{Q}_k} b_Q$. Then $b$ can be rewritten as $b=\sum\limits_{j\in \mathbb{Z}}B_j$.
Taking a smooth radial nonnegative function $\phi$ on $\R^d$ such that $\supp\ \phi\subset\{x:\frac{1}{2}\leq |x|\leq 2\}$ and
$\sum_j\phi_j(x)=1$ for all $x\in \mathbb{R}^d\backslash\{0\}$, where $\phi_j(x)=\phi(2^{-j}x)$.
Define the operator $T_j$ as
\Be\label{e:2dyda}
T_jh(x)=\int_{\R^d}\Om(x-y)\phi_j(x-y)K(x,y)h(y)dy.
\Ee
Then  $T_{\Om}=\sum\limits_jT_j$. For simplicity, we set $K_j(x,y)=\phi_j(x-y)K(x,y)$.
We write
\[T_{\Om}b(x)=\sum_{n\in\mathbb{Z}}\sum_{j\in\mathbb{Z}}T_jB_{j-n}.\]
Note that $T_jB_{j-n}(x)=0$ if $x\in (E^*)^c$ and $n<100$. Therefore
\begin{equation*}
\begin{split}
m\big(\big\{x\in (E^*)^c:&\,|T_{\Om}b(x)|>\frac{\lambda}{2}\big\}\big)\\
&=m\bigg(\bigg\{x\in (E^*)^c:\bigg|\sum_{n\geq100}\sum_{j\in
\mathbb{Z}}T_jB_{j-n}(x)\bigg|>\frac{\lambda}{2}\bigg\}\bigg).\\
\end{split}
\end{equation*}
Hence, to finish the proof of
of Theorem \ref{t:9main}, it suffices to verify the following estimate:
\begin{equation}\label{e:e^c}
\begin{split}
m\bigg(\bigg\{x\in (E^*)^c:\bigg|\sum_{n\geq100}\sum_{j\in
\mathbb{Z}}T_jB_{j-n}(x)\bigg|>\frac{\lambda}{2}\bigg\}\bigg)\lc {\lambda}^{-1}\C_\Om\|f\|_{1}.\\
\end{split}
\end{equation}
\vskip1cm

\subsection{Some key estimates}\label{s:82}\quad

Some important estimates play key roles in the proof of \eqref{e:e^c}. We present them by some lemmas, which will be proved in Section \ref{s:93} and Section \ref{s:94}.
The first estimate shows that the operator $T_j$ can be approximated by an operator $T_j^n$ in measure, which is defined below.

Let $l_\del(n)=[2\del^{-1}\log_2n]+2$. Here $[a]$ is the integer part of $a$. Let $\eta$ be a nonnegative, radial $C_c^\infty$ function which is supported in $\{|x|\leq1\}$ and $\int_{\R^d}\eta(x)dx=1$. Set $\eta_i(x)=2^{-id}\eta(2^{-i}x)$. Define
$$K_j^n(x,y)=\int_{\R^d}\eta_{j-l_\del(n)}(x-z)K_j(z,y)dz.$$

Notice that $K_j(z,y)$ is supported in $\{2^{j-1}\leq|z-y|\leq 2^{j+1}\}$ and $\eta_{j-l_\del(n)}(x)$ is supported in $\{|x|\leq2^{j-l_\del(n)}\}$, so $K_j^n(x,y)$ is supported in $\{2^{j-2}\leq|x-y|\leq2^{j+2}\}$. Therefore
\Be\label{e:8kjnb}
|K_j^n(x,y)|\lc 2^{-jd}\chi_{\{2^{j-2}\leq|x-y|\leq2^{j+2}\}}.
\Ee
Define the operator $T_j^n$ by
$$T_j^nh(x)=\int_{\R^d}\Om(x-y)K_j^n(x,y)\cdot h(y)dy.$$

\begin{lemma}\label{l:app}
With the notations above, we have
$$m\Big(\Big\{x\in (E^*)^c:\sum\limits_{n\geq100}\Big|
\sum\limits_j\big(T_jB_{j-n}(x)-T_j^nB_{j-n}(x)\big)\Big|>\fr{\lam}{4}\Big\}\Big)
\lc \frac{1}{\lambda}\|\Om\|_1\|f\|_{1}.$$
\end{lemma}

By  Lemma \ref{l:app}, the proof of \eqref{e:e^c}  now is reduced to verify the following estimate:
\Be\label{e:thm}
m\bigg(\bigg\{x\in (E^*)^c:
\bigg|\sum_{n\geq100}\sum_{j\in \mathbb{Z}}T_j^nB_{j-n}(x)\bigg|>\frac{\lambda}{4}\bigg\}\bigg)\lc
{\lambda}^{-1}\C_\Om\|f\|_{1}.
\Ee

Our second lemma shows that,  (\ref{e:thm}) holds if $\Om$ is restricted in some subset of $\S^{d-1}$. More precisely, for fixed $n\ge100$, denote $D^\iota=\{\theta\in\S^{d-1}:\,|\Om(\theta)|\geq2^{\iota n}\|\Om\|_1\}$, where $\iota>0$ will be chosen later.
The operator $T_{j,\iota}^n$ is defined by
$$T_{j,\iota}^nh(x)=\int_{\R^d}\Om\chi_{D^{\iota}}(\fr{x-y}{|x-y|})
K_j^n(x,y)\cdot h(y)dy.$$
We have the following result.
\begin{lemma}\label{l:2cur}
Under the conditions of Theorem \ref{t:9main}, for $f\in L^1(\R^d)$,  we have
\Bes
m\bigg(\bigg\{x\in (E^*)^c:
\bigg|\sum_{n\geq100}\sum_{j\in \mathbb{Z}}T_{j,\iota}^nB_{j-n}(x)\bigg|>
\frac{\lambda}{8}\bigg\}\bigg)\lc\C_\Om\fr{\|f\|_1}{\lam}.
\Ees
\end{lemma}

Thus, by Lemma \ref{l:2cur}, to finish the proof of Theorem \ref{t:9main}, it suffices to verify (\ref{e:thm}) under the condition that the kernel function $\Omega$ satisfies $\|\Om\|_\infty\leq2^{\iota n}\|\Om\|_1$ in each $T_j^{n}$.

In the following, we need to make a
microlocal decomposition of the kernel. To do this, we give a partition of unity on the unit surface $\S^{d-1}$.
Choose $n\geq100$. Let $\Theta_n=\{e^n_v\}_v$ be a collection of unit vectors on $\S^{d-1}$
which satisfies the following two conditions:

(a)\ $|e^n_v-e^n_{v'}|\geq 2^{-n\ga-4}$,  if $v\neq v'$;

(b)\ If $\theta\in \S^{d-1}$, there exists an $e^n_v$ such that $|e^n_v-\theta|\leq 2^{-n\ga-4}$.

\noindent The constant $0<\ga<1$ in (a) and (b) will be chosen later. To choose such an $\Theta_n$, we may simply take a maximal collection $\{e^n_v\}_v$ for which (a) holds. Notice that there are $C2^{n\ga(d-1)}$
elements in the collection $\{e^n_v\}_v$. For every $\tet\in\S^{d-1}$,
there only exists finite $e^n_v$ such that $|e^n_v-\tet|\leq2^{-n\ga-4}$. Now we can construct an associated partition of unity
on the unit surface $\S^{d-1}$. Let $\zeta$ be a smooth, nonnegative, radial function with $\zeta(u)=1$ for
$|u|\leq \fr{1}{2}$ and $\zeta(u)=0$ for $|u|>1$. Set $$\tilde{\Ga}^n_v(\xi)=\zeta\Big(2^{n\ga}(\fr{\xi}{|\xi|}-e^n_v)\Big)$$
and define
$$\Ga^n_v(\xi)=\tilde{\Ga}^n_v(\xi)\Big(\sum\limits_{e^n_v\in\Theta_n}\tilde{\Ga}^n_v(\xi)\Big)^{-1}. $$ Then it is easy to see that $\Ga^n_v$ is homogeneous of degree 0 with
$$\sum\limits_v\Ga^n_v(\xi)=1, \text{ for all $\xi\neq0$ and all $n$}. $$
Now we define operator $T_j^{n,v}$ by
\Be\label{e:Def T_j^{n,v}}
T_j^{n,v}h(x)=\int_{\R^d}\Om(x-y)\Ga^n_v(x-y)\cdot K_j^n(x,y)\cdot h(y)dy.
\Ee
Therefore, we have
$$T_j^n=\sum\limits_vT_j^{n,v}.$$

In the sequel, we need to separate the phase into different directions. Hence we define a multiplier operator by
$$\widehat{G_{n,v}h}(\xi)=\Phi(2^{n\ga}\inn{e^n_v}{{\xi}/{|\xi|}})\hat{h}(\xi),$$
where $h$ is a Schwartz function and $\Phi$ is a smooth, nonnegative, radial function such that
$0\leq\Phi(x)\leq1$ and $\Phi(x)=1$ on $|x|\leq2$, $\Phi(x)=0$ on $|x|>4$.
Now we can split $T_j^{n,v}$ into two parts:
\begin{equation*}
  T_j^{n,v}=G_{n,v}T_j^{n,v}+(I-G_{n,v})T_j^{n,v}.
\end{equation*}

The following lemma gives the $L^2$ estimate involving $G_{n,v}T_j^{n,v}$, which will be proved in next section.

\begin{lemma}{\label{l:L^2}}
Let $n\geq100$. Suppose $\|\Om\|_\infty\leq2^{\iota n}\|\Om\|_1$ in $T_j^n$, then we have the following estimate
$$\Big\|\sum\limits_j\sum\limits_vG_{n,v}T_j^{n,v}B_{j-n}\Big\|^2_2\lc2^{-n\ga+2n\iota}\lam\|\Om\|_1\|f\|_1.$$
\end{lemma}

The terms involving $(I-G_{n,v})T_j^{n,v}$ are more complicated. For convenience, we set $L^{n,v}_j=(I-G_{n,v})T_j^{n,v}$.
In Section \ref{s:94}, we shall prove the following lemma.
\begin{lemma}\label{l:L^1}
Suppose $\|\Om\|_\infty\leq2^{\iota n}\|\Om\|_1$ in $T_j^{n}$. With the notations above, we have
$$m\Big(\Big\{x\in (E^*)^c: \Big|\sum_{n\geq100}\sum\limits_j\sum_vL_j^{n,v}B_{j-n}(x)\Big|>\fr{\lam}{8}\Big\}\Big)\lc\lam^{-1}\|\Om\|_1\|f\|_1.$$
\end{lemma}

\subsection {Proof of \eqref{e:thm}.}\quad

We now complete the proof of \eqref{e:thm} under the condition $\|\Om\|_\infty\leq2^{\iota n}\|\Om\|_1$ in each $T_j^n$.  By Chebyshev's inequality,
\begin{equation*}
\begin{split}
&\ \ \ \ m\Big(\Big\{x\in (E^*)^c:\Big|\sum\limits_{n\geq100}\sum\limits_jT_j^{n}B_{j-n}(x)\Big|>\fr{\lam}{4}\Big\}\Big)\\
&\lc {\lam^{-2}}\Big\|\sum\limits_{n\geq100}\sum\limits_{j}\sum\limits_{v}G_{n,v}T_j^{n,v}B_{j-n}\Big\|_2^2
\\
&\ \ +m\Big(\Big\{x\in (E^*)^c: \Big|\sum\limits_{n\geq100}\sum\limits_{j}\sum\limits_{v}L_j^{n,v}B_{j-n}(x)\Big|>\fr{\lam}{8}\Big\}\Big)\\
&=:I+II.\\
\end{split}
\end{equation*}

Using Lemma \ref{l:L^1}, we can get the desired estimate of $II$. Next we consider the term $I$. Choose $0<\iota<\fr{\ga}{2}$. Minkowski's inequality and Lemma \ref{l:L^2} implies
\begin{equation*}
\begin{split}
I&\lc\lam^{-2}\Big(\sum\limits_{n\geq100}\Big\|\sum\limits_{j}\sum\limits_{v}G_{n,v}T_j^{n,v}B_{j-n}\Big\|_2\Big)^2\\
&\lc\lam^{-2}\Big(\sum\limits_{n\geq100}(2^{-n\ga+2n\iota}\|\Om\|_1\lam\|f\|_1)^{\fr{1}{2}}\Big)^{2}\lc \lam^{-1}\|\Om\|_1\|f\|_1.
\end{split}
\end{equation*}
We hence complete the proof of  Theorem \ref{t:9main} once Lemmas \ref{l:app}-\ref{l:L^1} hold.
\vskip1cm

\section {proofs of Lemmas \ref{l:app}-\ref{l:L^2}}\label{s:93}

\subsection {Proof of Lemma \ref{l:app}}\quad

We first focus on the proof of Lemma \ref{l:app}.
By the definitions of $T_j$ and $T_j^n$,
\begin{equation*}
\begin{split}
\|T_jf-T_j^nf\|_1&=\int_{\R^d}\Big|\int_{\R^d}\Om(x-y)(K_j(x,y)-K_j^n(x,y))f(y)dy\Big|dx\\
&= \int_{\R^d}\Big|\int_{\R^d}\Om(x-y)\int_{\R^d}\eta_{j-l_\del(n)}(z)(K_j(x,y)-K_j(x-z,y))dzf(y)dy\Big|dx.\\
\end{split}
\end{equation*}
By the definition of $K_j(x,y)$, we have
\Bes
|K_j(x,y)-K_j(x-z,y)|\leq|\phi_j(x-y)(K(x,y)-K(x-z,y))|+|\phi_j(x-y)-\phi_j(x-z-y)||K(x-z,y)|.
\Ees
Consider the first term firstly. Note that $|z|\leq2^{j-l_\del(n)}$ and $2^{j-1}\leq|x-y|\leq2^{j+1}$, then we have $2|z|<|x-y|$. By the regularity condition \eqref{e:8kr}, the first term above is bounded by
$$\fr{|z|^\del}{|x-y|^{d+\del}}\chi_{\{2^{j-1}\leq|x-y|\leq2^{j+1}\}}\lc n^{-2}2^{-jd}\chi_{\{2^{j-1}\leq|x-y|\leq2^{j+1}\}}.$$
We turn to the second therm. By the fact $|z|\leq2^{j-l_\del(n)}$ and the support of $\phi_j$, we have
$|x-y|\approx|x-z-y|$ and $2^{j-2}\leq|x-y|\leq2^{j+2}$. By \eqref{e:8kb}, the second term is controlled by
$$\fr{2^{-j}|z|}{|x-z-y|^d}\chi_{\{2^{j-2}\leq|x-y|\leq2^{j+2}\}}\lc n^{-2}2^{-jd}\chi_{\{2^{j-2}\leq|x-y|\leq2^{j+2}\}}.$$
Combining the above two estimates and applying Minkowski's inequality, we get
\Bes
\begin{split}
\|T_jf-T_j^nf\|_1&\lc n^{-2}\int_{\R^d}\int_{2^{j-2}\leq|x-y|\leq 2^{j+2}}2^{-jd}|\Om(x-y)|\int_{\R^d}\eta_{j-l_\del(n)}(z)dz|f(y)|dydx\\
&\lc n^{-2}2^{-jd}\int_{\R^d}\int_{2^{j-2}\leq|x-y|\leq2^{j+2}}|\Om(x-y)|dx|f(y)|dy\\
&\lc n^{-2}\|\Om\|_{1}\|f\|_1.
\end{split}
\Ees
By Chebyshev's inequality, Minkowski's inequality and the estimates above, we get the bound
\begin{equation*}
\begin{split}
m&\Big(\Big\{x\in (E^*)^c:\sum\limits_{n\geq100}\Big|\sum\limits_{j}
T_jB_{j-n}(x)-T_j^nB_{j-n}(x)\Big|>\fr{\lam}{4}\Big\}\Big)\\
&\lc\lam^{-1}\|\Om\|_1\sum\limits_{n\geq100}\sum\limits_j\Big\|T_jB_{j-n}-T_j^nB_{j-n}\Big\|_1\\
&\lc\lam^{-1}\|\Om\|_1\sum_{n\geq100}n^{-2}\sum\limits_j\|B_{j-n}\|_1\lc\lam^{-1}\|\Om\|_1\|f\|_1,
\end{split}
\end{equation*}
which is the required estimate.
$\hfill{} \Box$

\subsection {Proof of Lemma \ref{l:2cur}}\quad

Denote the kernel of the operator $T_{j,\iota}^n$ by $$K_{j,\iota}^n(x,y):=\Om\chi_{D^{\iota}}(\fr{x-y}{|x-y|})K_j^n(x,y).$$
By \eqref{e:8kjnb}, we have
$$\bigg|\int_{\R^d} K_{j,\iota}^n(x,y)dy\bigg|\lc\int_{D^{\iota}}\int_{2^{j-2}}^{2^{j+2}}|\Om(\theta)|r^{d-1}2^{-jd}drd\theta
\lc\int_{D^\iota}|\Om(\theta)|d\theta.$$
Therefore by Chebyshev's inequality, the above inequality, the property (cz-v), we get
\begin{equation*}
\begin{split}
m\bigg(\bigg\{&x\in (E^*)^c:
\bigg|\sum_{n\geq100}\sum_{j\in \mathbb{Z}}T_{j,\iota}^nB_{j-n}(x)\bigg|>\frac{\lambda}{8}\bigg\}\bigg)\\&
\lc{\lam}^{-1}\bigg\|\sum_{n\geq100}\sum_{j\in \mathbb{Z}}T_{j,\iota}^nB_{j-n}\bigg\|_{1}\lc{\lam}^{-1}\sum_{n\geq100}\sum_{j}\|B_{j-n}\|_1\int_{D^{\iota}}|\Om(\theta)|d\theta\\&
\lc{\lam}^{-1}\|b\|_1\int_{\S^{d-1}}\card\big\{n\in \N:\, n\geq 100, 2^{\iota n}\leq |\Om(\tet)|/\|\Om\|_1\big\}|\Om(\tet)|d\tet\\
&\lc{\lam}^{-1}\|f\|_1\int_{\S^{d-1}}|\Om(\tet)|\big(1+\log^+(|\Om|/\|\Om\|_1)\big)d\tet\\
&\lc{\lam}^{-1}\C_\Om\|f\|_1.
\end{split}
\end{equation*}
$\hfill{} \Box$

\subsection {Proof of Lemma \ref{l:L^2}}\label{s:933}\quad

 We will use some ideas from \cite{See96} in the proof of Lemma \ref{l:L^2}.
As usually, we adopt the $TT^*$ method in the $L^2$ estimate. Moreover, we need to use some orthogonality
argument based on the following observation of the support of $\mathcal{F}(G_{n,v}T_j^{n,v})$: For a fixed $n\geq 100$, we have
\begin{equation}\label{e:obser}
\sup\limits_{\xi\neq0}\sum\limits_{v}|\Phi^2(2^{n\ga}\inn{e^n_v}{\xi/|\xi|})|\lc2^{n\ga(d-2)}.
\end{equation}
In fact, by homogeneous of $\Phi^2(2^{n\ga}\inn{e^n_v}{\xi/|\xi|})$, it suffices to take the supremum over the surface $\S^{d-1}$. For $|\xi|=1$ and $\xi\in\supp\ \Phi^2(2^{n\ga}\inn{e^n_v}{\xi/|\xi|})$, denote by $\xi^{\bot}$ the hyperplane perpendicular to $\xi$. Thus
\begin{equation}\label{e:e^n_v}
\text{dist}(e^n_v,\xi^\bot)\lc2^{-n\ga}.
\end{equation}
Since the mutual distance of $e^n_v$'s is bounded by $2^{-n\ga-4}$, there are at most $2^{n\ga(d-2)}$ vectors satisfy
(\ref{e:e^n_v}). We hence get (\ref{e:obser}).

By applying Plancherel's theorem and Cauchy-Schwarz inequality, we have
\begin{equation}\label{e:L^2key}
\begin{split}
\Big\|\sum\limits_{v}\sum\limits_jG_{n,v}T^{n,v}_jB_{j-n}\Big\|^2_2
&=\Big\|\sum\limits_{v}\Phi(2^{n\ga}\inn{e^n_v}{\xi/|\xi|})\mathcal{F}\Big(\sum\limits_jT^{n,v}_jB_{j-n}\Big)(\xi)\Big\|^2_2\\
&\lc2^{n\ga(d-2)} \Big\|\sum\limits_{v}\Big|\mathcal{F}\Big(\sum\limits_jT^{n,v}_jB_{j-n}\Big)\Big|^2\Big\|_1\\
&\lc2^{n\ga(d-2)} \sum\limits_{v}\Big\|\sum\limits_jT^{n,v}_jB_{j-n}\Big\|^2_2.\\
\end{split}
\end{equation}
Once it is showed that for a fixed $e^n_v$,
\begin{equation}\label{e:L^2}
\Big\|\sum\limits_jT^{n,v}_jB_{j-n}\Big\|^2_2\lc2^{-2n\ga(d-1)+2n\iota}\lam\|\Om\|_1\|f\|_1,
\end{equation}
then by $\card(\Theta_n)\lc2^{n\ga(d-1)}$, and apply (\ref{e:L^2key}) and (\ref{e:L^2}) we get
\begin{equation*}
\Big\|\sum\limits_{v}\sum\limits_jG_{n,v}T^{n,v}_jB_{j-n}\Big\|^2_2\lc2^{-n\ga(d-1)-n\ga+2n\iota}\text{card($\Theta_n$)}\lam\|\Om\|_1\|f\|_1
\lc2^{-n\ga+2n\iota}\lam\|\Om\|_1\|f\|_1,
\end{equation*}
which is just the desired bound of Lemma \ref{l:L^2}. Thus, to finish the proof of Lemma \ref{l:L^2}, it is enough to prove  (\ref{e:L^2}).
By applying $\|\Om\|_\infty\leq 2^{\iota n}\|\Om\|_1$, \eqref{e:8kjnb} and the support of $\Ga_v^n$, we have
\begin{equation*}
\begin{split}
|T_j^{n,v}B_{j-n}(x)|&\lc2^{\iota n}\|\Om\|_1\int_{\R^d}\Ga^n_v(x-y)|K_j^n(x,y)||B_{j-n}(y)|dy\\
&\lc2^{\iota n}\|\Om\|_1H_j^{n,v}*|B_{j-n}|(x),
\end{split}
\end{equation*}
where $H_j^{n,v}(x):=2^{-jd}\chi_{E^{n,v}_j}(x)$ and $\chi_{E^{n,v}_j}(x)$ is a characteristic function of the set
$$E_j^{n,v}:=\{x\in \R^d:|\inn{x}{e^n_v}|\leq 2^{j+2},|x-\inn{x}{e^n_v}e^n_v|\leq 2^{j+2-n\ga}\}.$$
For a fixed $e^n_v$, we write
\begin{equation}\label{e:L^2key2}
\begin{split}
\Big\|\sum\limits_jT^{n,v}_j&B_{j-n}\Big\|^2_2\leq 2^{2\iota n}\|\Om\|^2_1\sum\limits_j\int_{\R^d} H_j^{n,v}*H_j^{n,v}*|B_{j-n}|(x)\cdot|B_{j-n}(x)|dx\\
&+2^{2\iota n+1}\|\Om\|^2_1\sum\limits_j\sum\limits_{i=-\infty}^{j-1}\int_{\R^d} H_j^{n,v}*H_i^{n,v}*|B_{i-n}|(x)\cdot|B_{j-n}(x)|dx.
\end{split}
\end{equation}
 Observe that $\|H_i^{n,v}\|_1\lc2^{-id}m(E_i^{n,v})\lc2^{-n\ga(d-1)}$, therefore for any $i\leq j$,
$$H_j^{n,v}*H_i^{n,v}(x)\lc 2^{-n\ga(d-1)}2^{-jd}\chi_{\widetilde{E}^{n,v}_j},$$
where $\widetilde{E}^{n,v}_j=E^{n,v}_j+E^{n,v}_j$.
Hence for a fixed $j$, $n$, $e^n_v$ and $x$, we have
\begin{equation}\label{e:L^2con}
\begin{split}
&H_j^{n,v}*H_j^{n,v}*|B_{j-n}|(x)+2\sum\limits_{i=-\infty}^{j-1}H_j^{n,v}*H_i^{n,v}*|B_{i-n}|(x)\\
&\lc2^{-n\ga(d-1)}2^{-jd}\sum\limits_{i\leq j}\int_{x+\widetilde{E}^{n,v}_j}|B_{i-n}(y)|dy\\
&\lc2^{-n\ga(d-1)}2^{-jd}\sum\limits_{i\leq j}\sum\limits_{Q\in\mathfrak{Q}_
{i-n}\atop Q\cap\{x+\widetilde{E}^{n,v}_j\}\neq\emptyset}\int_{\R^d}|b_Q(y)|dy\\
&\lc2^{-n\ga(d-1)}2^{-jd}\sum\limits_{i\leq j}\sum\limits_{Q\in\mathfrak{Q}_
{i-n}\atop Q\cap\{x+\widetilde{E}^{n,v}_j\}\neq\emptyset}\fr{\lam}{\C_\Om}|Q|\\
&\lc2^{-n\ga(d-1)}2^{-jd}2^{jd-n\ga(d-1)}\fr{\lam}{\C_\Om}=\fr{\lam}{\C_\Om} 2^{-2n\ga(d-1)},
\end{split}
\end{equation}
where in third inequality above, we use $\int|b_Q(y)|dy\lc\lam|Q|/{\C_\Om}$ (see (cz-v) in Section \ref{s:92}) and
in the fourth inequality we use fact that the cubes in $\mathcal{Q}$ are disjoint (see (cz-iii) in Section \ref{s:92}). By (\ref{e:L^2key2}),
 (\ref{e:L^2con}) and $\sum\limits_j\|B_{j-n}\|_1\lc\|f\|_1$, we obtain
$$\Big\|\sum\limits_jT^{n,v}_jB_{j-n}\Big\|^2_2\lc\lam2^{-2n\ga(d-1)+2n\iota}\|\Om\|_1\sum\limits_{j}\|B_{j-n}\|_1\lc\lam2^{-2n\ga(d-1)+2n\iota}\|\Om\|_1\|f\|_1.$$
Hence, we complete the proof of Lemma \ref{l:L^2}.
$\hfill{} \Box$
\vskip1cm

\section{Proof of Lemma \ref{l:L^1}}\label{s:94}

To prove Lemma \ref{l:L^1}, we have to face with some oscillatory
integrals which come from $L_j^{n,v}$.
We first introduce Mihlin multiplier theorem, which can be found in \cite{Gra249}.

\begin{lemma}\label{l:9mihlin}
Let $m$ be a complex-value bounded function on $\R^n\setminus\{0\}$ that satisfies
$$|\pari_\xi^\alpha m(\xi)|\leq A|\xi|^{-|\alpha|}$$
for all multi indices $|\alpha|\leq [\fr{d}{2}]+1$, then the operator $T_m$ defined by
$$\widehat{T_mf}(\xi)=m(\xi)\hat{f}(\xi)$$
is a weak type {\rm(1,1)} bounded operator with bound $C_d(A+\|m\|_\infty)$.
\end{lemma}

Before stating the proof of Lemma \ref{l:L^1}, let us give some notations. We first introduce the Littlewood-Paley decomposition. Let $\psi$ be a radial
 $C^\infty$ function such that $\psi(\xi)=1$ for $|\xi|\leq 1$, $\psi(\xi)=0$ for $|\xi|\geq 2$
and $0\leq\psi(\xi)\leq1$ for all $\xi\in\R^d$. Define $\beta_k(\xi)=\psi(2^k\xi)-\psi(2^{k+1}\xi)$,
then $\beta_k$ is supported in $\{\xi:2^{-k-1}\leq|\xi|\leq 2^{-k+1}\}$. Define the convolution operators $V_k$ and $\Lam_k$
with Fourier multipliers $\psi(2^k\cdot)$ and $\beta_k$, respectively. That is,
$$\widehat{V_kf}(\xi)=\psi(2^k\xi)\hat{f}(\xi),\ \ \widehat{\Lam_kf}(\xi)=\beta_k(\xi)\hat{f}(\xi).$$
Then by the construction of $\beta_k$ and $\psi$, we have
$$I=\sum\limits_{k\in\mathbb{Z}}\Lam_k=V_m+\sum\limits_{k<m}\Lam_k \quad \text{for every}\ m\in\Z.$$
Set $A_{j,m}^{n,v}=V_mT_j^{n,v}$ and $D_{j,k}^{n,v}=(I-G_{n,v})\Lam_{k}T_{j}^{n,v}$. Write
\begin{equation*}
\begin{split}
L_j^{n,v}
&=(I-G_{n,v})V_mT_j^{n,v}+\sum\limits_{k<m}(I-G_{n,v})\Lam_{k}T_{j}^{n,v}\\
&=:(I-G_{n,v})A_{j,m}^{n,v}+\sum\limits_{k<m}D_{j,k}^{n,v},
\end{split}
\end{equation*}
where $m=j-[n\eps_0]$, $\eps_0>0$ will be chosen later. To prove Lemma \ref{l:L^1}, we split the measure in Lemma \ref{l:L^1} into two parts,
\begin{equation}\label{e:l^1}
\begin{split}
&\ \ \ \ m\Big(\Big\{x\in (E^*)^c: \Big|\sum_{n\geq100}\sum_v\sum\limits_j(I-G_{n,v})T_j^{n,v}B_{j-n}(x)\Big|>\lam\Big\}\Big)\\
&\leq m\Big(\Big\{x\in (E^*)^c:\Big|\sum_{n\geq100}\sum_v(I-G_{n,v})\big(\sum\limits_jA_{j,m}^{n,v}B_{j-n}\big)(x)\Big|
>\fr{\lam}{2}\Big\}\Big)\\
&\ \ \ \ +m\Big(\Big\{x\in (E^*)^c:\Big|\sum_{n\geq100}\sum_v\sum\limits_j\sum\limits_{k<m}D_{j,k}^{n,v}B_{j-n}(x)\Big|
>\fr{\lam}{2}\Big\}\Big)\\
&=:I+II.
\end{split}
\end{equation}

\subsection {First step: basic estimates of $I$ and $II$}\quad

Consider the term $I$. Notice that $\mathcal{F}[(I-G_{n,v})f](\xi)=(1-\Phi(2^{n\ga}\inn{e^n_v}{\xi/|\xi|}))\cdot\hat{f}(\xi)$.
It is easy to see that $(1-\Phi(2^{n\ga}\inn{e^n_v}{\xi/|\xi|}))$ is bounded and
$$|\pari_{\xi}^{\alpha}(1-\Phi(2^{n\ga}\inn{e^n_v}{\xi/|\xi|}))|\lc2^{n\ga([\fr{d}{2}]+1)}|\xi|^{-|\alpha|}$$
for all multi indices $|\alpha|\leq [\fr{d}{2}]+1$.
Then by Lemma \ref{l:9mihlin},  $I-G_{n,v}$ is of weak type (1,1) with bound $2^{n\ga([\fr{d}{2}]+1)}.$
By using the pigeonhole principle, one may get
\Be\label{e:9pid}
\{x:\sum\limits_{i}f_i(x)>\sum\limits_{i}\lam_i\}\subseteq\bigcup_{i}\{x:f_i(x)>\lam_i\}.
\Ee
Let $\mu>0$ to be chosen later. Then there exists $C_{\mu,d}$ such that $$\sum\limits_{n\geq100}\sum\limits_{e^n_v\in\Theta_n}C_{\mu,d}2^{-n\mu-n\ga(d-1)}=\fr{1}{2}.$$ Therefore
\begin{equation}\label{e:L^1_1}
\begin{split}
&\ \ \ \ m\Big(\Big\{x\in (E^*)^c:\Big|\sum_{n\geq100}\sum_v(I-G_{n,v})\Big(\sum\limits_jA_{j,m}^{n,v}B_{j-n}\Big)(x)\Big|
>\fr{\lam}{2}\Big\}\Big)\\
&= m\Big(\Big\{x\in (E^*)^c:\Big|\sum_{n\geq100}\sum_v(I-G_{n,v})\Big(\sum\limits_jA_{j,m}^{n,v}B_{j-n}\Big)(x)\Big|
>\sum_{n\geq100}\sum_vC_{\mu,d}2^{-n\mu-n\ga(d-1)}\lam\Big\}\Big)\\
&\leq\sum_{n\geq100}\sum_vm\Big(\Big\{x\in (E^*)^c:\Big|(I-G_{n,v})\Big(\sum\limits_jA_{j,m}^{n,v}B_{j-n}\Big)(x)\Big|
>C_{\mu,d}2^{-n\mu-n\ga(d-1)}\lam\Big\}\Big)\\
&\leq\sum_{n\geq100}\sum_j\sum_v \fr{1}{C_{\mu,d}\lam}2^{n\mu+n\ga(d-1)+n\ga([\fr{d}{2}]+1)}\|A_{j,m}^{n,v}B_{j-n}\|_1\\
&\leq\sum_{n\geq100}\sum_j\sum_v\sum_{l(Q)=2^{j-n}}
\fr{1}{C_{\mu,d}\lam}2^{n\mu+n\ga(d-1)+n\ga([\fr{d}{2}]+1)}\|A_{j,m}^{n,v}b_{Q}\|_1,
\end{split}
\end{equation}
where the second inequality follows from \eqref{e:9pid} and in the third inequality we use $I-G_{n,v}$ is weak type (1,1) bounded and Minkowski's inequality.

Next we turn to the term $II$. We use $L^1$ estimate directly
\begin{equation}\label{e:L^1_2}
\begin{split}
II\leq \fr{2}{\lam}\sum_{n\geq100}\sum_v\sum\limits_j\sum\limits_{k<m}\|D_{j,k}^{n,v}B_{j-n}\|_1
\leq \fr{2}{\lam}\sum_{n\geq100}\sum_v\sum\limits_j\sum\limits_{k<m}
\sum\limits_{l(Q)=2^{j-n}}\|D_{j,k}^{n,v}b_Q\|_1.
\end{split}
\end{equation}

Now the problem is reduced to estimate $\|A_{j,m}^{n,v}b_Q\|_1$ and $\|D_{j,k}^{n,v}b_Q\|_1$.
Recall in \eqref{e:Def T_j^{n,v}}, the kernel of operator  $T^{n,v}_j$ is
$$K^{n,v}_{j,y}(x):=\Om(x-y)\Ga_v^n(x-y)K_j^{n}(x,y).$$

Below we see $K^{n,v}_{j,y}(x)$ as a function of $x$ for a fixed $y\in Q$.
Thus, by Fubini's theorem,
$$A_{j,m}^{n,v}b_Q(x)=\int_QV_mK^{n,v}_{j,y}(x)\cdot b_Q(y)dy
=:\int_Q A_m(x,y)b_Q(y)dy$$
and
$$D_{j,k}^{n,v}b_Q(x)=\int_Q(I-G_{n,v})\Lambda_{k}K^{n,v}_{j,y}(x)\cdot b_Q(y)dy
=:\int_Q D_{k}(x,y)b_Q(y)dy.$$

\subsection {Estimate of $D_{k}$}\quad

\begin{lemma}\label{l:l^1_1}
For a fixed $y\in Q$, there exists $N>0$, such that for any $N_1\in\Z_+$
\Be\label{e:main}
\|D_{k}(\cdot,y)\|_1\leq Cn^{2\del^{-1}N_1}2^{-n\ga(d-1)+n\iota}2^{(-j+k)N_1+n\ga(N_1+2N)}\|\Om\|_1,
\Ee
where $C$ is a constant independent of $y$, but may depend on $N_1$, $N$ and $d$.
\end{lemma}

\begin{proof}
Denote $h_{k,n,v}(\xi)=(1-\Phi(2^{n\ga}\inn{e^n_v}{\xi/|\xi|}))\beta_{k}(\xi).$ Write $D_k(x,y)$ as
\begin{equation*}
(I-G_{n,v})\Lam_{k}K^{n,v}_{j,y}(x)=\fr{1}{(2\pi)^{d}}
\int_{\R^d} e^{ix\cdot\xi}h_{k,n,v}(\xi)\int_{\R^d} e^{-i\xi\cdot\om}\Om(\om-y)\Ga_v^n(\om-y)K^{n}_j(\om,y) d\om d\xi.
\end{equation*}
In order to separate the rough kernel, we make a variable change $\om-y=r\theta$. By Fubini's theorem, the integral above can be written as
\begin{equation}\label{e:mainintegral}
\fr{1}{(2\pi)^{d}}\int_{\S^{d-1}}\Om(\theta)\Ga_v^n(\theta)
\bigg\{\int_{\R^d}\int_0^\infty
e^{i\inn{x-y-r\theta}{\xi}}h_{k,n,v}(\xi)K_j^n(y+r\tet,y)r^{d-1} drd\xi\bigg\}d\tet.
\end{equation}
By the support of $K_j^n(x,y)$ in \eqref{e:8kjnb}, we have $2^{j-2}\leq r\leq2^{j+2}$. Integrate by parts $N_1$ times with $r$. Hence the integral involving $r$ can
be rewritten as
$$\int_0^\infty e^{i\inn{x-y-r\theta}{\xi}}(i\inn{\theta}{\xi})^{-N_1}
\pari^{N_1}_r[K_j^n(y+r\tet,y)r^{d-1}]dr.$$

Since $\theta\in\supp\ \Ga^n_v$, then $|\theta-e^n_v|\leq 2^{-n\ga}$. By the support of $\Phi$,
we see $|\inn{e^n_v}{\xi/|\xi|}|\geq 2^{1-nr}$. Thus,
\begin{equation}\label{e:2ang}
|\inn{\theta}{\xi/|\xi|}|\geq|\inn{e^n_v}{\xi/|\xi|}|-|\inn{e^n_v-\theta}{\xi/|\xi|}|\geq2^{-n\ga}.
\end{equation}
After integrating by parts with $r$, integrate by parts with $\xi$, the integral in \eqref{e:mainintegral}
can be rewritten as
\begin{equation}\label{e:minte}
\begin{split}
\fr{1}{(2\pi)^d}\int_{\S^{d-1}}&\Om(\tet)\Ga_v^n(\tet)\int_{\R^d}e^{i\inn{x-y-r\tet}{\xi}}\int_0^\infty
\pari_r^{N_1}\Big(K_j^n(y+r\tet,y)r^{d-1}\Big)\times\\
&\fr{(I-2^{-2k}\Delta_\xi)^N}{(1+2^{-2k}|x-y-r\tet|^2)^N}\Big(h_{k,n,v}(\xi)(i\inn{\tet}{\xi})^{-N_1}\Big)drd\xi d\tet.
\end{split}
\end{equation}

In the following, we give an explicit estimate of the term in (\ref{e:minte}). By the definition of $K_j^n(x,y)$, we have
\Be\label{e:8parix}
\begin{split}
|\pari_x^\alp K_j^n(x,y)|&=2^{-(j-l_\del(n))|\alp|}\Big|\int(\pari_x^\alp\eta)_{j-l_\del(n)}(x-z)K_j(z,y)dz\Big|\\
&\leq2^{-(j-l_\del(n))|\alp|}\|K_j(\cdot,y)\|_\infty\|\pari_x^\alp\eta\|_1\\
&\lc2^{-(j-l_\del(n))|\alp|-jd},
\end{split}
\Ee
where the third inequality follows from \eqref{e:8kjnb}. By using product rule,
\begin{equation}\label{e:8parir}
\begin{split}
\Big|\pari^{N_1}_r\Big(K_j(y+r\tet,y)r^{d-1}\Big)\Big|&=\Big|\sum_{i=0}^{N_1}C_{N_1}^i\pari_r^i(K_j^n(y+r\tet,y))\pari_r^{N_1-i}(r^{d-1})\Big|\\
&=\Big|\sum\limits_{i=N_1-d+1}^{N_1}C_{N_1}^i\pari_r^i(K_j^n(y+r\tet,y))\pari_r^{N_1-i}(r^{d-1})\Big|.\\
\end{split}
\end{equation}
Applying \eqref{e:8parix} and $2^{j-2}\leq r\leq2^{j+2}$ , the above \eqref{e:8parir} is bounded by
\Be\label{e:8rbound}
\sum\limits_{i=N_1-d+1}^{N_1}C_{N_1}^i2^{-(j-l_\del(n))i-jd}2^{(j+2)(d-1-N_1+i)}\\
\leq C_{N_1}n^{2\del^{-1}N_1}2^{-(1+N_1)j}.
\Ee

Below we will show that
\begin{equation}\label{e:2D2}
\big|(I-2^{-2k}\Delta_\xi)^{N}[\inn{\theta}{\xi}^{-N_1}h_{k,n,v}(\xi)]\big|\leq C_{N_1}2^{(n\ga+k)N_1+2n\ga N}.
\end{equation}
We prove \eqref{e:2D2} when $N=0$ firstly. By \eqref{e:2ang}, we have
$$|(-i\inn{\theta}{\xi})^{-N_1}\cdot h_{k,n,v}(\xi)|\lc|\inn{\theta}{\xi}|^{-N_1}\lc2^{(n\ga+k)N_1}.$$
By using product rule,
\begin{equation*}
|\pari_{\xi_i}h_{k,n,v}(\xi)|=\big|-\pari_{\xi_i}[\Phi(2^{n\ga}\inn{e^n_v}{\xi/|\xi|})]
\cdot\beta_{k}(\xi)+\pari_{\xi_i}\beta_{k}(\xi)\cdot
(1-\Phi(2^{n\ga}\inn{e^n_v}{\xi/|\xi|}))\big|
\lc2^{n\ga+k}.
\end{equation*}
Therefore by induction, we have
$|\pari_{\xi}^{\alpha}h_{k,n,v}(\xi)|\lc2^{(n\ga+k)|\alpha|}$
for any  multi-indices $\alpha\in\Z^n_+$.
By using product rule again and (\ref{e:2ang}), we have
\begin{equation*}
\begin{split}
\big|\pari^2_{\xi_i}(\inn{\theta}{\xi})^{-N_1}h_{k,n,v}(\xi))\big|&=\big|{\inn{\theta}{\xi}}^{-N_1-2}\cdot N_1(N_1+1)\theta_i^2\cdot h_{k,n,v}\\
&\quad+2{\inn{\theta}{\xi}}^{-N_1-1}\cdot(-N_1)\cdot\theta_i\pari_{\xi_i}h_{k,n,v}(\xi)
+\inn{\theta}{\xi}^{-N_1}\pari_{\xi_i}^2h_{k,n,v}(\xi)\big|\\
&\leq C_{N_1}2^{(n\ga+k)(N_1+2)}.\\
\end{split}
\end{equation*}
Hence we conclude that
\begin{equation*}
2^{-2k}\big|\Delta_\xi[(\inn{\theta}{\xi})^{-N_1}h_{k,n,v}(\xi)]\big|\leq C_{N_1}2^{(n\ga+k)N_1+2n\ga}.
\end{equation*}
Proceeding by induction, we get \eqref{e:2D2}.

Now we choose $N=[d/2]+1$.
Since we need to get the $L^1$ estimate of (\ref{e:mainintegral}), by the support of $h_{k,n,v}$,
$$\int_{\supp(h_{k,n,v})}\int\Big(1+2^{-2k}|x-y-r\theta|^2\Big)^{-N}dxd\xi\leq C.$$
Integrating with $r$, we get a bound $2^j$. Note that we suppose that $\|\Om\|_\infty\leq 2^{n\iota}\|\Om\|_1$. Then integrating with $\theta$, we get a bound $2^{-n\ga(d-1)+n\iota}\|\Om\|_1$. Combining (\ref{e:8rbound}), (\ref{e:2D2}) and above estimates, $\|D_{k}(\cdot,y)\|_1$ is bounded by
\Bes
\begin{split}
&\ \ \ \ C_{N_1}n^{2\del^{-1}N_1}2^{-j(1+N_1)+(n\ga+k)N_1+2n\ga N+j-n\ga(d-1)+n\iota}\|\Om\|_1\\
&=C_{N_1}n^{2\del^{-1}N_1} 2^{-n\ga(d-1)+n\iota}2^{(-j+k)N_1+n\ga(N_1+2N)}\|\Om\|_1.\\
\end{split}
\Ees
Hence we complete the proof of Lemma \ref{l:l^1_1} with $N=[\fr{d}{2}]+1$.
\end{proof}

\subsection {Estimate of $A_{m}$.}\quad

Using the cancellation condition of $b_Q$ (see (cz-v) in Section \ref{s:92}), we have
$$A_{j,m}^{n,v}b_Q(x)=\int_Q(A_{m}(x,y)-A_m(x,y_0))b_Q(y)dy,$$
where $y_0$ is the center of $Q$. By changing to polar coordinates and applying Fubini's theorem, we can write $A_m(x,y)$ as
\begin{equation*}
\fr{1}{(2\pi)^{d}}\int_{\S^{d-1}}\Om(\theta)\Ga_v^n(\theta)
\bigg\{\int_0^\infty\int_{\R^d}
e^{i(\inn{x-y-r\theta}{\xi}}\psi(2^m\xi)K_j^n(y+r\tet,y)r^{d-1}drd\xi\bigg\} d\tet.
\end{equation*}
Integrating by part $N=[d/2]+1$ times with $\xi$ in the above integral, we have
\Bes
\begin{split}
\fr{1}{(2\pi)^{d}}\int_{\S^{d-1}}\Om(\theta)\Ga_v^n(\theta)
\bigg\{\int_0^\infty\int_{\R^d}
&e^{i\inn{x-y-r\theta}{\xi}}K_j^n(y+r\tet,y)r^{d-1}\\
&\times\fr{(I-2^{-2m}\Delta_\xi)^{N}\psi(2^m\xi)}
{\big(1+2^{-2m}|x-y-r\theta|^2\big)^{N}}d\xi dr
\bigg\}d\theta.
\end{split}
\Ees
Denote
$$A_{m}(x,y)-A_{m}(x,y_0)=: F_{m,1}(x,y)+F_{m,2}(x,y)+F_{m,3}(x,y),$$
where
\begin{equation*}
\begin{split}
F_{m,1}(x,y)=
\fr{1}{(2\pi)^{d}}\int_{\S^{d-1}}&\Om(\theta)\Ga_v^n(\theta)
\bigg\{\int_0^\infty\int_{\R^d}
\Big(e^{i\inn{-y}{\xi}}-e^{i\inn{-y_0}{\xi}}\Big)e^{i\inn{x-r\theta}{\xi}}\\
&\times K_j^n(y+r\tet,y)r^{d-1}\fr{(I-2^{-2m}\Delta_\xi)^{N}\psi(2^m\xi)}
{\big(1+2^{-2m}|x-y-r\theta|^2\big)^{N}}d\xi dr
\bigg\}d\theta,
\end{split}
\end{equation*}
\begin{equation*}
\begin{split}
F_{m,2}(x,y)=\fr{1}{(2\pi)^{d}}\int_{\S^{d-1}}\Om(\theta)\Ga_v^n(\theta)
\bigg\{\int_0^\infty\int_{\R^d}
&e^{i\inn{x-y_0-r\theta}{\xi}}\Big(K_j^n(y+r\tet,y)-K_j^n(y_0+r\tet,y_0)\Big)\\
&\times r^{d-1}\fr{(I-2^{-2m}\Delta_\xi)^{N}\psi(2^m\xi)}
{\big(1+2^{-2m}|x-y-r\theta|^2\big)^{N}}d\xi dr
\bigg\}d\theta,
\end{split}
\end{equation*}
and
\begin{equation*}
\begin{split}
F_{m,3}(x,y)&=\fr{1}{(2\pi)^{d}}\int_{\S^{d-1}}\Om(\theta)\Ga_v^n(\theta)\int_0^\infty\int_{\R^d}
e^{i\inn{x-y_0-r\theta}{\xi}}(I-2^{-2m}\Delta_\xi)^{N}\psi(2^m\xi)r^{d-1}\times\\
&K_j^n(y_0+r\tet,y_0)\Big(\fr{1}{\big(1+2^{-2m}|x-y-r\theta|^2\big)^{N}}-\fr{1}{\big(1+2^{-2m}|x-y_0-r\theta|^2\big)^{N}}\Big)d\xi drd\theta.
\end{split}
\end{equation*}
Hence
\begin{equation}\label{e:A_m}
\|A_{j,m}^{n,v}b_Q\|_1\leq \sup_{y\in Q}(\|F_{m,1}(\cdot,y)\|_1+\|F_{m,2}(\cdot,y)\|_1+\|F_{m,3}(\cdot,y)\|_1)\|b_Q\|_1.
\end{equation}
We have the following estimates of $F_{m,1}(x,y)$, $F_{m,2}(x,y)$, $F_{m,3}(x,y)$.
\begin{lemma}\label{l:l^1_2}
For a fixed $y\in Q$, we have
\Bes
\|F_{m,1}(\cdot,y)\|_1\leq C2^{-n\ga(d-1)+n\iota+j-n-m}\|\Om\|_1,
\Ees
where $C$ is independent of $y$.
\end{lemma}
\begin{proof}
We use the same method in proving Lemma \ref{l:l^1_1} but don't apply integrating by parts. Note that $y\in Q$ and $y_0$
is the center of $Q$, then $|y-y_0|\lc 2^{j-n}$. Thus
$$\Big|e^{i\inn{-y}{\xi}}-e^{i\inn{-y_0}{\xi}}\Big|\lc 2^{j-n-m}.$$
Since $2^{j-2}\leq r\leq2^{j+2}$ and \eqref{e:8kjnb}, we have
$|K_j^n(y+r\tet,y)r^{d-1}|\lc 2^{-j}.$
It is easy to see that
$$|(I-2^{-2m}\Delta_\xi)^N\psi(2^m\xi)|\leq C.$$
Since we need to get the $L^1$ estimate of $F_{m,1}(\cdot,y)$, by the support of $\psi(2^m\xi)$, we have
$$\int_{|\xi|\leq2^{1-m}}\int\Big(1+2^{-2m}|x-y-r\theta|^2\Big)^{-N}dxd\xi\leq C.$$
Integrating with $r$, we get a bound $2^j$. Note that we suppose that $\|\Om\|_\infty\leq 2^{n\iota}\|\Om\|_1$, so integrating with $\theta$, we get a bound $2^{-n\ga(d-1)+n\iota}\|\Om\|_1$.
Combining these bounds, we can get the required estimate for $F_{m,1}(\cdot,y)$.
\end{proof}
\begin{lemma}\label{l:8fm3}
For a fixed $y\in Q$, we have
\Bes
\|F_{m,3}(\cdot,y)\|_1\leq C2^{-n\ga(d-1)+n\iota+j-n-m}\|\Om\|_1,
\Ees
where $C$ is independent of $y$.
\end{lemma}
\begin{proof}
For the term $F_{m,3}(\cdot,y)$, we can deal with it in the same way as $F_{m,1}(\cdot,y)$ once we
have the following observation
\begin{equation*}
\begin{split}
\Big|\Psi(y)-\Psi(y_0)\Big|&=\Big|\int_0^1\big\langle y-y_0,
\nabla \Psi(ty+(1-t)y_0)\big\rangle dt
\Big|\\
&\lc|y-y_0|2^{-m}\int_0^1\fr{N2^{-m}|x-(ty+(1-t)y_0)-r\theta|}{(1+2^{-2m}|x-(ty+(1-t)y_0)-r\theta|^2)^{N+1}}dt
\end{split}
\end{equation*}
where $\Psi(y)=(1+2^{-2m}|x-y-r\theta|^2)^{-N}$. Since $y\in Q$ and $y_0$ is the center of $Q$, we have $|y-y_0|\lc 2^{j-n}$. By $2^{j-2}\leq r\leq2^{j+2}$ and \eqref{e:8kjnb}, we have
$|K_j^n(y+r\tet,y)r^{d-1}|\lc 2^{-j}.$  It is easy to see
$$|(I-2^{-2m}\Delta_\xi)^N\psi(2^m\xi)|\leq C.$$
Since we need to get the $L^1$ estimate of $F_{m,3}(\cdot,y)$, by the support of $\psi(2^m\xi)$, we have
$$\int_{|\xi|\leq2^{1-m}}\int\fr{N2^{-m}|x-(ty+(1-t)y_0)-r\theta|}{(1+2^{-2m}|x-(ty+(1-t)y_0)-r\theta|^2)^{N+1}}dxd\xi\leq C.$$
Integrating with $r$, we get a bound $2^j$. Integrating with $t$, we get finite bound $1$. Note that we suppose that $\|\Om\|_\infty\leq 2^{n\iota}\|\Om\|_1$, therefore integrating with $\theta$, we get a bound $2^{-n\ga(d-1)+n\iota}\|\Om\|_1$.
Combining these bounds, we can get the required estimate for $F_{m,3}(\cdot,y)$.
\end{proof}

\begin{lemma}\label{l:L^1_3}
For a fixed $y\in Q$, we have
\Bes
\|F_{m,2}(\cdot,y)\|_1\leq C\Big(n^{2\del^{-1}}2^{-n}+2^{-n\del}\Big)2^{-n\ga(d-1)+n\iota}\|\Om\|_1,
\Ees
where $C$ is independent of $y$.
\end{lemma}
\begin{proof}
First, notice that $2^{j-2}\leq r\leq2^{j+2}$. Write $K_j^n(y+r\tet,y)-K_j^n(y_0+r\tet,y_0)$ as
$$\Big(K_j^n(y+r\tet,y)-K_j^n(y_0+r\tet,y)\Big)+\Big(K_j^n(y_0+r\tet,y)-K_j^n(y_0+r\tet,y_0)\Big).$$
Since $y\in Q$ and $y_0$ is the center of $Q$, we have $|y-y_0|\leq2^{j-n}$. Therefore by the mean value formula, Minkowski's inequality and \eqref{e:8kjnb}, we get
\Be\label{e:8amk1}
\begin{split}
&\ \ \ \ \Big|K_j^n(y+r\tet,y)-K_j^n(y_0+r\tet,y)\Big|\\
&=
\Big|\int_{\R^d}\Big(\eta_{j-l_\del(n)}(y+r\tet-z)-\eta_{j-l_{\del}(n)}{(y_0+\tet-z)}\Big)K_j(z,y)dz\Big|\\
&=\Big|\int_{\R^d}\Big(\int_0^1\inn{y-y_0}{\nabla(\eta_{j-l_\del(n)})(ty+(1-t)y_0+r\tet-z)}dt\Big)K_j(z,y)dz\Big|\\
&\leq|y-y_0|2^{-j+l_\del(n)}\sum_{i=1}^n\|\pari_{x_i}\eta\|_1\|K_j(\cdot,y)\|_\infty\\
&\lc n^{2\del^{-1}}2^{-n-jd}.
\end{split}
\Ee
We write
\Be\label{e:8amk2}
\begin{split}
&\Big|K_j^n(y_0+r\tet,y)-K_j^n(y_0+r\tet,y_0)\Big|\\
&=\Big|\int_{\R^d} \eta_{j-l_\del(n)}(y_0+r\tet-z)\Big(K_j(z,y)-K_j(z,y_0)\Big)dz\Big|\\
&\leq\Big|\int_{\R^d} \eta_{j-l_\del(n)}(y_0+r\tet-z)\Big(\phi_j(z-y)-\phi_j(z-y_0)\Big)K(z,y)dz\Big|\\
&\ \ +\Big|\int_{\R^d} \eta_{j-l_\del(n)}(y_0+r\tet-z)\Big(K(z,y)-K(z,y_0)\Big)\phi_j(z-y_0)dz\Big|\\
&=:P_1+P_2.
\end{split}
\Ee
Consider $P_1$ firstly. Using the fact $|y-y_0|\lc2^{j-n}$ and the support of $\phi$, we have $2^{j-2}\leq|z-y|\leq2^{j+2}$. Applying the mean value formula, we get
$$P_1\leq |y-y_0|2^{-j}\|K(\cdot,y)\|_\infty\|\eta\|_1\lc 2^{-n-jd}.$$

For the term $P_2$, by $|y-y_0|<2^{j-n}$ and $2^{j-1}\leq|z-y_0|\leq2^{j+1}$, we have $2|y-y_0|\leq|z-y_0|$. By the regularity condition \eqref{e:8kr}, we have
$$P_2\leq C\int_{2^{j-2}\leq|z-y_0|\leq2^{j+2}}\eta_{j-l_\del(n)}(y_0+r\tet-z)\fr{|y-y_0|^\del}{|z-y_0|^{d+\del}}dz\lc2^{-n\del-jd}.$$

Combining the estimates of $P_1$ and $P_2$, we conclude that \eqref{e:8amk2} is controlled by $2^{-n\del-jd}$. Now we come back to estimate the $L^1(\R^d)$ norm of $F_{m,2}(\cdot,y)$.
It is easy to check
$$|(I-2^{-2m}\Delta_\xi)^N\psi(2^m\xi)|\leq C.$$
Since we need to get the $L^1$ estimate of $F_{m,2}(\cdot,y)$, by the support of $\psi(2^m\xi)$, we have
$$\int_{|\xi|\leq2^{1-m}}\int\Big(1+2^{-2m}|x-y-r\theta|^2\Big)^{-N}dxd\xi\leq C.$$
Integrating with $r$, we get
$$\int_{2^{j-2}}^{2^{j+2}}r^{d-1}dr\approx 2^{jd}.$$
Integrating with $\theta$, we get a bound $2^{-n\ga(d-1)+n\iota}\|\Om\|_1$. Combining with the estimates in \eqref{e:8amk1} and \eqref{e:8amk2}, the $L^1$ norm of $F_{m,2}(\cdot,y)$ is bounded by
$$\Big(n^{2\del^{-1}}2^{-n}+2^{-n\del}\Big)2^{-n\ga(d-1)+n\iota}\|\Om\|_1,$$
which is the required bound.
\end{proof}

\subsection{Proof of Lemma \ref{l:L^1}}\quad

Let us come back to the proof of Lemma \ref{l:L^1}, it is sufficient to consider $I$ and $II$ in (\ref{e:l^1}).
By (\ref{e:L^1_1}), (\ref{e:L^1_2}) and (\ref{e:A_m}), we have
\begin{equation*}
\begin{split}
I+II&\leq \fr{2}{\lam}\sum_{n\geq100}\sum_j\sum_v\sum_{l(Q)=2^{j-n}}
\Big[C_{\mu,d}^{-1}2^{n\mu+n\ga(d-1)+n\ga([\fr{d}{2}]+1)}
\|A_{j,m}^{n,v}b_{Q}\|_1+
\sum_{k<m}\|D_{j,k}^{n,v}b_Q\|_1\Big]\\
&\leq \fr{2}{\lam}\sum_{n\geq100}\sum_j\sum_v\sum_{l(Q)=2^{j-n}}\sup_{y\in Q}
\Big[C_{\mu,d}^{-1}2^{n\mu+n\ga(d-1)+n\ga([\fr{d}{2}]+1)}\Big(\|F_{m,1}(\cdot,y)\|_1\\
&\quad+\|F_{m,2}(\cdot,y)\|_1+\|F_{m,3}(\cdot,y)\|_1\Big) +\sum_{k<m}\|D_{k}(\cdot,y)\|_1\Big]\|b_{Q}\|_1.
\end{split}
\end{equation*}

Notice $m=j-[n\eps_0]$ and $\card(\Theta_n)\lc2^{n\ga(d-1)}$. Now applying Lemma \ref{l:l^1_1} with $N=[\fr{d}{2}]+1$, then Lemma \ref{l:l^1_2}, Lemma \ref{l:8fm3}, Lemma \ref{l:L^1_3} and the fact $[n\eps_0]\leq n\eps_0<[n\eps_0]+1$ imply
\begin{equation*}
\begin{split}
I+II\lc{\lam}^{-1}\sum_{n\geq100}\sum_j\sum_{l(Q)=2^{j-n}}\|b_Q\|_1\|\Om\|_1
\big[C_{\mu,d}^{-1}(2^{s_1n}+
n^{2\del^{-1}}2^{s_2n}+2^{s_3n})+n^{2\del^{-1}N_1}2^{s_4n}\big],
\end{split}
\end{equation*}
where
\begin{equation*}
\begin{split}
s_1&=\mu+\ga(d-1)+\ga\big([\fr{d}{2}]+1\big)-1+\eps_0+\iota,\\
s_2&=\mu+\ga(d-1)+\ga\big([\fr{d}{2}]+1\big)-1+\iota,\\
s_3&=\mu+\ga(d-1)+\ga\big([\fr{d}{2}]+1\big)-\del+\iota,\\
s_4&=-\eps_0N_1+\ga N_1+2\big([\fr{d}{2}]+1\big)\ga+\iota.
\end{split}
\end{equation*}
Now we choose $0<\iota\ll\ga\ll\eps_0\ll1$, $0<\mu\ll\del$, $0<\ga\ll\del$, $0<\iota\ll\del$ and $N_1$ large enough
such that $$\max\{s_1,s_2,s_3,s_4\}<0.$$
Therefore $$I+II\lc\fr{\|\Om\|_1}{\lam}\|b\|_1\sum_{n\geq100}\big[C_{\mu,d}^{-1}(2^{s_1n}+
n^{2\del^{-1}}2^{s_2n}+2^{s_3n})+n^{2\del^{-1}N_1}2^{s_4n}\big]\lc\fr{\|\Om\|_1}{\lam}\|f\|_1.$$ Hence we finish the proof of Lemma \ref{l:L^1}, thus we prove Theorem \ref{t:9main}.$\hfill{} \Box$
\vskip1cm

\section{Applications of the criterion}\label{s:92l}

In this section, we will give some important and interesting applications of Theorem \ref{t:9main}. Notice the following well known embedding relations between some function spaces on $ \S^{d-1}$:
\begin{equation*}\label{including rel}
L^\infty(\S^{d-1})\subsetneq L^r(\S^{d-1})\,(1<r<\infty)\subsetneq L\log^+\!\!L(\S^{d-1})\subsetneq L^1(\S^{d-1}),
\end{equation*}
and $\|\Omega\|_{L\log^+L}\lc \|\Omega\|_{r}$ when $\Omega\in L^r(\S^{d-1})\,(1<r\le\infty).$
Thus, we may get the following corollary of Theorem \ref{t:9main}:

\begin{corollary}\label{coro of main}
Suppose $K$ satisfies \eqref{e:8kb} and \eqref{e:8kr}. Let $\Om$ satisfy \eqref{e:2Home} and $\Om\in L^r(\S^{d-1})$ for $1<r\le\infty$. In addition, suppose $\Om$ and $K$ satisfy some appropriate cancellation conditions such that $T_\Om f(x)$ in \eqref{e:8T} is well defined for $f\in C_c^\infty(\R^d)$ and maps $L^2(\R^{d})$ to itself with bound $\|\Om\|_{r}$. Then for any $\lam>0$, we have
$$\lam m(\{x\in\R^{d}:|T_\Om f(x)|>\lam\})\lc\C_{\Om,r}\|f\|_1$$
where $\mathcal{C}_{\Om,r}=\|\Om\|_{r}+\int_{\S^{d-1}}|\Om(\tet)|\big(1+\log^+({|\Om(\tet)|}/{\|\Om\|_{1}})\big)d\tet$.
\end{corollary}

Obviously, the weak type (1,1) bounds of rough singular integral $T$ given in Theorem B are immediate consequences of applying Theorem \ref{t:9main}. In fact, it is easy to see that $$K(x,y)=\fr{1}{|x-y|^d}$$ in the kernel of the singular integral $T$ defined in \eqref{e:9Tsin} satisfies \eqref{e:8kb} and \eqref{e:8kr} with $\del=1$.

In the following we give some applications of Theorem \ref{t:9main} and Corollary \ref{coro of main} involving Calder\'on commutator and its generalizations, which arises naturally in the studies of the Cauchy integral on Lipschitz curve and differential equations with non-smooth coefficients, see \cite{Cal78}, \cite{Fef74}, \cite{MC97} and \cite{MS13} for the background and applications of Calder\'on commutator.

\subsection{Calder\'on commutator}\quad

Recall Calde\'on commutator defined in \eqref{TA},
\Bes
T_{\Om,A} f(x)=\pv\int_{\R^d}\fr{\Om(x-y)}{|x-y|^d}\cdot\fr{A(x)-A(y)}{|x-y|}\cdot f(y)dy,
\Ees
As a first application of Theorem \ref{t:9main}, we get the weak type (1,1) boundedness of Calder\'on commutator $T_{\Om,A}$.
\begin{theorem}\label{t:9main_1}
Suppose $\Om\in L\log^+\!\!L(\S^{d-1})$ satisfying \eqref{e:2Home} and \eqref{e:km} and $A\in Lip(\R^d)$. Then for any $\lam>0$, we have
\begin{equation*}
m(\{x\in\R^d:|T_{\Om,A}f(x)|>\lam\})\lc{\lam}^{-1}\C_\Om\|\nabla A\|_\infty\|f\|_1.
\end{equation*}
\end{theorem}
\begin{proof}
Under the conditions in Theorem \ref{t:9main_1} , by Theorem C, we know that $T_\Om$ is bounded on $L^2(\R^d)$ with bound $\|\nabla A\|_\infty\|\Om\|_{L\log^+L}$. Hence, to prove the Theorem \ref{t:9main_1}, by Theorem \ref{t:9main}, it is enough to show that the kernel
$$K(x,y)=\fr{1}{|x-y|^d}\fr{A(x)-A(y)}{|x-y|}$$
satisfies \eqref{e:8kb} and \eqref{e:8kr}. Since $A\in Lip(\R^d)$, it is trivial to see that \eqref{e:8kb} holds. Suppose $|x_1-y|>2|x_1-x_2|$, then we have $|x_1-y|\approx|x_2-y|$. Applying the mean value formula, we have
\Bes
\begin{split}
|K(x_1,y)-K(x_2,y)|&\leq\Big|\fr{1}{|x_1-y|^{d+1}}-\fr{1}{|x_2-y|^{d+1}}\Big||A(x_1)-A(y)|+\fr{|A(x_1)-A(x_2)|}{|x_2-y|^{d+1}}\\
&\lc\|\nabla A\|_\infty\fr{|x_1-x_2|}{|x_1-y|^{d+1}}.
\end{split}
\Ees
Thus the first inequality in \eqref{e:8kr} is valid. The proof of the second inequality in \eqref{e:8kr} is similar. Hence we complete the proof.
\end{proof}

\subsection{Higher order Calder\'on commutator}\quad

In 1990, S. Hofmann \cite{Hof90} gave the $L^p\,(1<p<\infty)$ boundedness of the higher order Calder\'on commutator defined by
\begin{equation}\label{TAk}
T^k_{\Om,A} f(x)=\pv\int_{\R^d}\fr{\Om(x-y)}{|x-y|^d}\cdot\bigg(\fr{A(x)-A(y)}{|x-y|}\bigg)^k\cdot f(y)dy,
\end{equation}
where  $\Omega$ satisfies \eqref{e:2Home}, $A\in Lip(\R^d)$ and $k\ge1$.

\vskip0.2cm
\noindent
\textbf{Theorem D}\ {\rm (\cite{Hof90}).}\ {\it Suppose that $\Omega\in L^\infty(\S^{d-1})$
and satisfies the moment conditions
\begin{equation}\label{moment cond}
\int_{\S^{d-1}}\Om(\tet)\tet^\alpha d\tet=0,\quad \text{for all}\ \ \alpha\in \Z_+^d\ \ \text{with}\ \ |\alp|=k.
\end{equation}
Then the higher order Calder\'on commutator $T^k_{\Om,A}$ defined in \eqref{TAk} is a bounded operator on $L^p(\R^d)$ for $1<p<\infty$ with bound $\|\Omega\|_{\infty}\|\nabla A\|_{\infty}^k$.}
\vskip0.2cm

Applying Corollary \ref{coro of main}, we show that the higher order Calder\'on commutator $T^k_{\Om,A}$ is of weak type (1,1).
\begin{theorem}\label{t:main_1}
Suppose that $k\ge1$, $\Om\in L^\infty(\S^{d-1})$ satisfying \eqref{e:2Home} and \eqref{moment cond} and $A\in Lip(\R^d)$. Then for any $\lam>0$, we have
\begin{equation*}
m(\{x\in\R^d:|T^k_{\Om,A}f(x)|>\lam\})\lc{\lam}^{-1}\|\Omega\|_\infty\|\nabla A\|_{\infty}^k\|f\|_1.
\end{equation*}
\end{theorem}
\begin{proof} The proof is similar to the proof of Theorem \ref{t:9main_1}.
By Corollary \ref{coro of main} and Theorem D, it only needs to check that the kernel
$$K(x,y)=\fr{1}{|x-y|^d}\bigg(\fr{A(x)-A(y)}{|x-y|}\bigg)^k$$
satisfies \eqref{e:8kb} and \eqref{e:8kr}.
On one hand, the verification of \eqref{e:8kb}  is trivial since $A\in Lip(\R^d)$. On the other hand,  if  $|x_1-y|>2|x_1-x_2|$, we have $|x_1-y|\approx|x_2-y|$. Applying the mean value formula, we get
\Bes
\begin{split}
&|K(x_1,y)-K(x_2,y)|\\
&\leq\Big|\fr{1}{|x_1-y|^{d}}-\fr{1}{|x_2-y|^{d}}\bigg|\bigg|\fr{A(x_1)-A(y)}{|x_1-y|}\bigg|^k\\
&\quad+\frac 1{|x_2-y|^{d}}\bigg|\bigg(\fr{A(x_1)-A(y)}{|x_1-y|}\bigg)^k
-\bigg(\fr{A(x_2)-A(y)}{|x_2-y|}\bigg)^k\bigg|\\
&\lc\|\nabla A\|^k_\infty\fr{|x_1-x_2|}{|x_1-y|^{d+1}}.
\end{split}
\Ees
Thus the first inequality in \eqref{e:8kr} is valid. The proof of the second inequality in \eqref{e:8kr} is similar. Hence we complete the proof.
\end{proof}

\subsection{General Calder\'on commutator}\quad

In \cite{Cal77}, Calder\'on introduce the following more general commutator
\Be\label{e:9gc}
T_{\Om,F,A}f(x)=\pv\int_{\R^d}\fr{\Om(x-y)}{|x-y|^d}F\Big(\fr{A(x)-A(y)}{|x-y|}\Big)f(y)dy.
\Ee
It is well known that the study of this commutator is closely connected to the Cauchy integral on Lipschitz curves and the elliptic boundary value problem on non-smooth domain (see \cite{Cal78}, \cite{Cal77}, \cite{CCFJR78} and \cite{FJR78}). In \cite{CCFJR78}, by using the method of rotation, A. P. Calder\'on \emph{et al.}pointed that

\vskip0.2cm
\noindent
\textbf{Theorem E}\ {\rm (\cite{CCFJR78}).}\ {\it Suppose  $\Omega$, $F$ and $A$ satisfy the following conditions, then the commutator $T_{\Om,F,A}$ defined in \eqref{e:9gc} is bounded on $L^p(\R^d)$ for $1<p<\infty$:

{\rm (i)}\ $\Om(-\tet)=-\Om(\tet)$ for $\tet\in\S^{d-1}$ and $\Om\in L^1(\S^{d-1})$;

{\rm (ii)}\ $A\in Lip(\R^d)$ ;

{\rm (iii)}\ $F(t)=F(-t)$ for $t\in\R$ and $F(t)$ is real analytic in $\{|t|\leq\|\nabla A\|_\infty\}$.}
\vskip0.2cm

 \noindent  Using Theorem \ref{t:9main}, we may get a weak type (1,1) boundedness of $T_{\Om,F,A}$.
\begin{theorem}
Suppose $\Om$, $A$ and $F$ satisfy the conditions {\rm (i)$\sim$(iii)} in Theorem E. If  $\Om\in L\log^+L(\S^{d-1})$, then the general Calder\'on commutator $T_{\Om,F,A}$ is of weak type $(1,1)$. That is,  for any $\lam>0$
and $f\in L^1$,  $$m(\{x\in\R^d:|T_{\Om,F,A}f(x)|>\lam\})\lc{\lam}^{-1}\C_\Om\|f\|_1.$$
\end{theorem}
\begin{proof}
By Theorem \ref{t:9main} and Theorem E, it is enough to show that the kernel
$$K(x,y)=\fr{1}{|x-y|^d}F\Big(\fr{A(x)-A(y)}{|x-y|}\Big)$$
satisfies \eqref{e:8kb} and \eqref{e:8kr}. It is easy to check that $$|K(x,y)|\leq\fr{1}{|x-y|^d}\|F\|_{L^\infty(B(0,\|\nabla A\|_\infty))}.$$ Suppose $|x_1-y|>2|x_1-x_2|$, then $|x_1-y|\approx|x_2-y|$. Using the mean value formula and the fact $F$ is analytic in $\{|t|\leq\|\nabla A\|_\infty\}$, we have
\Bes
\begin{split}
|K(x_1,y)-K(x_2,y)|&\leq\Big|\fr{1}{|x_1-y|^d}-\fr{1}{|x_2-y|^d}\Big|\Big|F\Big(\fr{A(x_1)-A(y)}{|x_1-y|}\Big)\Big|\\
&\ \ \ \ +\fr{1}{|x_2-y|^d}\Big|F\Big(\fr{A(x_1)-A(y)}{|x_1-y|}\Big)-F\Big(\fr{A(x_2)-A(y)}{|x_2-y|}\Big)\Big|\\
&\lc\fr{|x_1-x_2|}{|x_1-y|^{d+1}}\Big(\|F\|_{L^\infty(B(0,\|\nabla A\|_\infty))}+\|\nabla A\|_\infty\|\nabla F\|_{L^\infty(B(0,\|\nabla A\|_\infty))}\Big).
\end{split}
\Ees
Thus the first inequality in \eqref{e:8kr} is valid. Similarly we can establish the second inequality in \eqref{e:8kr}. Therefore we complete the proof.
\end{proof}

\subsection{Calder\'on commutator of Bajsanski-Coifman type}\quad

In 1967, Bajsanski and Coifman \cite{BC67} introduced another kind of general Calder\'on commutator as follows. For a multi-indices $\alpha\in \Z_+^d$, set $A_{\alpha}(x)=\pari^{\alp}_xA(x)$ and
$$P_l(A,x,y)=A(x)-\sum\limits_{|\alp|<l}\fr{A_{\alp}(y)}{\alp!}(x-y)^\alp,$$
where $l\in\Bbb N$. Define the singular operator $T_{\Omega,A,l}$ as
\begin{equation}\label{e:8ht}
T_{\Omega,A,l}f(x)=\pv\int_{\R^d}\fr{\Om(x-y)}{|x-y|^d}\cdot\fr{P_l(A,x,y)}{|x-y|^l}\cdot
f(y)dy,
\end{equation}
where $\Omega$ satisfies \eqref{e:2Home} and \eqref{int condi}. Clearly, when $l=1$, the operator $T_{\Omega,A,l}$ is just Calder\'on commutator $T_{\Omega,A}$ defined in \eqref{TA}.

\vskip0.2cm
\noindent
\textbf{Theorem F}\ {\rm (\cite{BC67}).}\ {\it The commutator $T_{\Omega,A,l}$ defined in \eqref{e:8ht} is bounded on $L^p(\R^d)$ for $1<p<\infty$ if  $l\in\Bbb N$ and $\Omega, A$ satisfy the following conditions:

{\rm (i)}\ $\Om\in L\log^+\!\!L(\S^{d-1})$ and satisfies \eqref{e:2Home} and
\begin{equation}\label{e:km_1}
\int_{\S^{d-1}}\Om(\tet)\tet^\alpha d\tet=0,\quad \text{for all}\ \ \alpha\in \Z_+^d\ \ \text{with}\ \ |\alp|=l;
\end{equation}

{\rm (ii)}\ $A_\alpha\in L^\infty(\R^d)$ for $|\alpha|=l$.}

\vskip0.2cm

E. M. Stein pointed out that the operator  $T_{\Omega,A,l}$ is of weak type $(1,1)$ if $\Om\in Lip(\S^{d-1})$.

\vskip0.2cm
\noindent
\textbf{Theorem G}\ {\rm (E. M. Stein, see \cite[p.\,16]{BC67}).}\ {\it Suppose  $l\in\Bbb N$ and $\Omega, A$ satisfy the same conditions as Theorem F, but replacing $\Om\in L\log^+\!\!L(\S^{d-1})$ by $\Om\in Lip(\S^{d-1})$,  then $T_{\Omega,A,l}$ is of weak type $(1,1)$.}
\vskip0.2cm

Applying Theorem \ref{t:9main}, we may improve Theorem G essentially.

\begin{theorem}\label{t:main_3}
Let $l\geq 1$. Suppose $\Om\in L\log^+\!\!L(\S^{d-1})$ satisfying \eqref{e:2Home} and \eqref{e:km_1}. Let $ A_\alp\in L^\infty(\R^d)$ for every $|\alp|=l$. Then for any $\lam>0$, we have
\begin{equation*}
m(\{x\in\R^d:|T_{\Omega,A,l}f(x)|>\lam\})\lc{\lam}^{-1}\C_\Om\sum_{|\alp|=l}\|A_\alp\|_\infty\|f\|_1.
\end{equation*}
\end{theorem}

\begin{remark} When $l=1$, $T_{\Omega,A,1}$ equals to $T_{\Om,A}$ defined in \eqref{TA}. Thus, Theorem \ref{t:9main_1} is just the special case of  Theorem \ref{t:main_3} when $l=1$.
\end{remark}
\begin{proof}
By Theorem \ref{t:9main} and Theorem F, to prove Theorem \ref{t:main_3}, it suffices to show that the
kernel
$$K(x,y)=\fr{1}{|x-y|^d}\cdot\fr{P_l(A,x,y)}{|x-y|^l}$$
satisfies \eqref{e:8kb} and \eqref{e:8kr}. By the fact $A_\alp\in L^\infty(\R^d)$ for every $|\alp|=l$ and the following Taylor expansion
$$P_{l}(A,x,y)=l\sum\limits_{|\alp|=l}\fr{(x-y)^\alp}{\alp!}\int_0^1(1-s)^{l-1}
A_{\alp}(y+s(x-y))ds,$$
we conclude that $$|K(x,y)|\lc\sum_{|\alp|=l}\|A_\alpha\|_\infty\fr{1}{|x-y|^d}.$$

Choose $|x_1-y|>2|x_1-x_2|$. Then we have $|x_1-y|\approx|x_2-y|$. By using the Taylor expansion, we can write
\Bes
\begin{split}
P_{l}(A,x,y)&=P_{l-1}(A,x,y)-\sum\limits_{|\alp|=l-1}\fr{A_{\alp}(y)}{\alp!}(x-y)^\alp\\
&=(l-1)\sum\limits_{|\alp|=l-1}\fr{(x-y)^\alp}{\alp!}\int_0^1(1-s)^{l-2}
\Big(A_{\alp}(y+s(x-y))-A_{\alp}(y)\Big)ds.
\end{split}
\Ees
Note that for each $|\alp|=l-1$, $A_\alp\in Lip(\R^d)$. By the mean value formula, it is not difficult to see that
$$|K(x_1,y)-K(x_2,y)|\lc\sum_{|\alp|=l}\| A_\alp\|_\infty\fr{|x_1-x_2|}{|x_1-y|^{d+1}}.$$
The proof of the second inequality in \eqref{e:8kr} is similar. Hence \eqref{e:8kr} holds for $K(x,y)$. Thus we finish the proof.
\end{proof}

\subsection{General singular integral of Muckenhoupt type}\quad

In 1960, B. Muckenhoupt \cite{Muc60} considered a modification of singular integral and generalized Calder\'on
and Zygmund's work \cite{CZ52} and \cite{CZ56} on the fractional integration in the following. Suppose that $\Om$ satisfies \eqref{e:2Home}$\sim$\eqref{int condi}. Then the following singular integral operator is well defined for $f\in C_c^\infty(\R^d)$ and $r\in\R\setminus\{0\}$,
\Be\label{e:9cs}
T_{\Om,ir}f(x)=\pv\int_{\R^d}\fr{\Om(x-y)}{|x-y|^{d+ir}}f(y)dy,
\Ee
where $i=\sqrt{-1}$.
\vskip0.2cm
\noindent
\textbf{Theorem H}\ {\rm (\cite[Theorem 8]{Muc60}).}\ {\it With the above definition of the general singular integral operator $T_{\Om,ir}$, $T_{\Om,ir}$ is bounded on $L^p(\R^d)$ with bound $C_r\|\Om\|_{1}$ for $1<p<\infty$. Here we should point out $\Om$ satisfies additional cancelation condition \eqref{cance condi} so that $T_{\Om,ir}f$ is well defined for $f\in C_c^\infty(\R^d)$.}
\vskip0.2cm

As a final application of Theorem \ref{t:9main}, we can establish the weak type (1,1) boundedness of $T_{\Om,ir}$.
\begin{theorem}\label{Muck}
Suppose $\Om$ satisfies \eqref{e:2Home}, \eqref{cance condi} and $\Om\in L\log^+L(\S^{d-1})$. Then for any $\lam>0$,
$$m(\{x\in\R^d:|T_{\Om,ir}f(x)|>\lam\})\lc{\lam}^{-1}\C_\Om\|f\|_1.$$
\end{theorem}
\begin{proof}
By Theorem \ref{t:9main} and Theorem H, it suffices to verify the kernel
$$K(x,y)=\fr{1}{|x-y|^{d+ir}}$$
satisfying \eqref{e:8kb} and \eqref{e:8kr}. It is easily to see that $|K(x,y)|=\fr{1}{|x-y|^d}$. Suppose $|x_1-y|>2|x_1-x_2|$, then $|x_1-y|\approx|x_2-y|$. By using the mean value formula, we have
\Bes
\begin{split}
&|K(x_1,y)-K(x_2,y)|\\
&\leq\Big|\fr{1}{|x_1-y|^d}-\fr{1}{|x_2-y|^d}\Big| +\fr{1}{|x_2-y|^d}\Big|e^{-ir\ln|x_1-y|}-e^{-ir\ln|x_2-y|}\Big|\\
&\lc\fr{|x_1-x_2|}{|x_1-y|^{d+1}}.
\end{split}
\Ees
So the first inequality in \eqref{e:8kr} is valid. Similarly we can establish the second inequality in \eqref{e:8kr}. Hence we complete the proof.
\end{proof}

\section{Some further problems}\label{s:76}

In the previous section, we give lots of applications of Theorem \ref{t:9main}. However, there are still many operators that do not fall into the scope of our main result's applications. Below we list some open problems related to weak type (1,1) bound (For more we refer the reader to see \cite{See14}, \cite{GSPSS17}).

\subsection{Oscillatory singular integral operator with rough kernel.} Let $P(x,y)$ be a real-valued polynomial on $\R^d\times\R^d$. S. Lu and Y. Zhang \cite{LZ92} showed that the operator defined by
$$Tf(x)=\pv\int_{\R^d}e^{iP(x,y)}\fr{\Om(x-y)}{|x-y|^d}f(y)dy$$
is bounded on $L^p(\R^d)(1<p<+\infty)$ if $\Om$ satisfies \eqref{e:2Home}, \eqref{cance condi} and $\Om\in L^r(\S^{d-1})(1<r\leq+\infty)$. S. Challino and M. Christ \cite{CC87} proved that this operator is of weak type (1,1) if $\Om\in Lip(\S^{d-1})$.
It is interesting to show $T$ is weak (1,1) bounded if $\Om$ is rough.

\subsection{Commutator of Christ-Journ\'e type.} Let $a\in L^\infty(\R^d)$, let $K$ be the Calder\'on-Zygmund convolution kernel. M. Christ and J. L. Journ\'e \cite{CJ87} proved the operator defined by
$$T_{a,k}f(x)=\pv\int_{\R^d}K(x-y)(m_{x,y}a)^kf(y)dy$$
maps $L^p(\R^d)$ to itself for $1<p<+\infty$, where $m_{x,y}a=\int_0^1a(sx+(1-s)y)ds$. A. Seeger \cite{See14} showed that $T_{a,1}$ is of weak type (1,1). It is open whether $T_{a,k}$ is weak (1,1) bounded for $k\geq2$. If replacing the Calder\'on-Zygmund convolution kernel $K(x)$ by $\Om(x)/|x|^d$ with $\Om$ is homogeneous of degree zero, S. Hofmann \cite{Hof95} proved this kind of operator maps $L^p(w)$ to itself for  $w$ an $A_p$ weight and $1<p<\infty$ if $\Om\in L^\infty(\S^{d-1})$. One can also ask a question whether it is weak type (1,1) bounded if $\Om\in L^\infty(\S^{d-1})$.

\subsection{Maximal singular integral operator with rough kernel.} Suppose $K$ satisfies \eqref{e:8kb} and \eqref{e:8kr}. Let $\Om$ satisfy \eqref{e:2Home} and $\Om\in L\log^+L(\S^{d-1})$. Suppose $\Om$ and $K$ satisfy some appropriate cancellation conditions such that the following operator
$$T_*f(x)=\sup_{\eps>0}\Big|\int_{|x-y|>\eps}\Om(x-y)K(x,y)f(y)dy\Big|.$$
is well defined for $f\in C_c^\infty(\R^d)$ and extends to a bounded operator on $L^2(\R^{d})$ with bound $C\|\Om\|_{L\log^+L}$. Then a natural question is that wether $T_*$ is of weak type (1,1). When $K(x,y)=1/|x-y|^d$, Calder\'on and Zygmund \cite{CZ56} showed that $T_*$ is $L^p(\R^d)$ bounded for $1<p<+\infty$ if $\Om\in L\log^+L(\S^{d-1})$. But it is unknown whether $T_*$ is of weak type (1,1) even when $\Om\in L^\infty(\S^{d-1})$. And when $K(x,y)=\fr{A(x)-A(y)}{|x-y|^{d+1}}$, $A$ is a Lipschitz function, A. P. Calder\'on \cite{Cal65} proved that $T_*$ is $L^p(\R^d)$ bounded for $1<p<+\infty$ if $\Om\in L\log^+L(\S^{d-1})$. Also the weak type (1,1) bound is unknown in this case.

\subsection*{Acknowledgements}
Xudong Lai would like to thank Andreas Seeger for some helpful
discussions related to this work. The authors thank the referee for his/her important comments and very valuable suggestions which improve this manuscript greatly.
\vskip1cm

\bibliographystyle{amsplain}

\end{document}